\documentclass[a4paper,11pt,leqno]{article}
\usepackage{amssymb,amsmath,amsthm,amscd,amsfonts}
\usepackage{amsmath,amssymb,amsthm,amscd}

\usepackage{graphicx}

\newtheorem{defi}{Definition}
\newtheorem{fact}{Fact}
\newtheorem{thm}{Theorem}[section]

\newtheorem{prop}{Proposition}[section]
\newtheorem{lemm}{Lemma}[section]

\newtheorem{example}{Example}
\numberwithin{equation}{section}
\theoremstyle{remark}
\newtheorem{remark}{Remark}[section]
\renewcommand{\tablename}{Table}
\renewcommand{\thethm}{\Alph{thm}}

\def\hsymb#1{\mbox{\strut\rlap{\smash{\Huge$#1$}}\quad}} 
\newcommand{\BigFig}[1]{\parbox{12pt}{\Huge #1}}
\newcommand{\BigZero}{\BigFig{0}} 
\def\rnum#1{\expandafter{\romannumeral #1}} 
\def\Rnum#1{\uppercase\expandafter{\romannumeral #1}} 

\newcounter{mynum}
\renewcommand{\themynum}{\arabic{mynum}.}
\newcommand{\myitem}{\refstepcounter{mynum}\mbox{\themynum}} 

\begin{document}
\title{\textbf{Pseudo-hyperbolic Gauss maps\\
of Lorentzian surfaces in anti-de Sitter space
}}
\author{Honoka Kobayashi and Naoyuki Koike}
\date{\today}
\maketitle

\begin{abstract}
In this paper, we determine the type numbers of the pseudo-hyperbolic Gauss maps of all oriented Lorentzian surfaces 
of constant mean and Gaussian curvatures and non-diagonalizable shape operator in the $3$-dimensional anti-de Sitter 
space.  Also, we investigate the behavior of type numbers of the pseudo-hyperbolic Gauss map along the parallel 
family of such oriented Lorentzian surfaces in the $3$-dimensional anti-de Sitter space. 
Furthermore, we investigate the type number of the pseudo-hyperbolic Gauss map of one of Lorentzian hypersurfaces 
of B-scroll type in a general dimensional anti-de Sitter space.  
\end{abstract}

\section{Introduction}

\hspace{11pt} The notion of finite typeness of isometric immersions into a Euclidean space was introduced by B.Y. Chen in the late 
1970's (see \cite{C1},\,\cite{C2},\,\cite{C3},\,\cite{CV} etc.).  
Later, the finite typeness of isometric immersions into a Euclidean space or, in more general, a pseudo-Euclidean 
space have been studied by many geometers (see \cite{C1},\,\cite{C2}, \cite{C3} etc.).  
B.Y. Chen and P. Piccinni (\cite{CP}) extended the notion of the finite typeness to $C^{\infty}$-maps and 
studied the finite typeness of the Gauss map (which is not an immersion) of isometric immersions.  
Recently B. Bekta\c{s}, E.\"O. Canfes, U. Dursun and R. Ye\v{g}in has been studied the finite typeness of 
the pseudo-spherical (resp. the pseudo-hyperbolic) Gauss maps of isometric immersions into the pseudo-sphere 
(resp. the pseudo-hyperbolic space) (see \cite{BD},\,\cite{BCD1},\,\cite{BCD2},\,\cite{YD}).  

Let $M$ be an $n$-dimensional pseudo-Riemannian manifold of index $t$ and $\mathbb{E}^m_s$ be an $m$-dimensional pseudo-Euclidean space of index $s$. 
The smooth map $\phi:M \rightarrow \mathbb{E}^m_s$ is said to be {\it of finite type} if $\phi$ has the spectral decomposition: 
$\phi=\phi_1+\cdots +\phi_k$ with $\phi_i:M\rightarrow \mathbb{E}^m_s$'s are non-constant map such that $\Delta \phi_i=\lambda_i\phi_i$, 
where $\Delta$ is the Laplacian operator of $M$ and $\lambda_i$ are constants, 
and furthermore, if $\lambda_i$'s are mutually distinct, then $\phi$ is said to be {\it of $k$-type}. 
Denote by $\mathbb{S}^{m-1}_{s}$ the $(m-1)$-dimensional pseudo-sphere of constant curvature $1$ and index $s$ and 
$\mathbb{H}^{m-1}_s$ the $(m-1)$-dimensional pseudo-hyperbolic space of constant curvature $-1$ and index $s$.  
Let $M$ be an $n$-dimensional oriented pseudo-Riemannian manifold of index $t$ and 
${\bf x}:M\hookrightarrow \mathbb{S}^{m-1}_{s}\subset \mathbb{E}^m_s$ be an isometric immersion.  
Let $G(n+1,m)_{t}$ be the Grassmannian manifold consisting of $(n+1)$-dimensional oriented non-degenerate subspaces 
of index $t$ of $\mathbb{E}^m_s$. 
Define a map $\tilde{\nu}:M\rightarrow G(n+1,m)_{t}$ by 
$\tilde{\nu}(p)={\bf x}(p)\wedge {\bf x}_*(e_1^p)\wedge {\bf x}_*(e_2^p)\wedge \cdots \wedge {\bf x}_*(e_n^p)$ for 
$p\in M$, 
where $(e_1^p, \cdots ,e_n^p)$ is an orthonormal frame of $T_pM$ compatible with the orientation of $M$. 
This map $\tilde{\nu}$ is called {\it the pseudo-spherical Gauss map of ${\bf x}$}.  
Similarly, for an isometric immersion ${\bf x}:M\hookrightarrow \mathbb{H}^{m-1}_s \subset \mathbb{E}^m_{s+1}$, 
{\it the pseudo-hyperbolic Gauss map}
$\tilde{\nu}:M\rightarrow G(n+1,m)_{t+1}$ is defined. 

D.S. Kim and Y.H. Kim(\cite{KimKim}) classified the Lorentzian surfaces of 
constant mean and Gaussian curvatures and non-diagonalizable shape operator in the $3$-dimensional 
de Sitter $\mathbb{S}^3_1$ and anti-de Sitter $\mathbb{H}^3_1$ space as follows.  
\begin{fact}
{\sl Let $M^2_1$ be a Lorentzian surface in $\mathbb{S}^3_1$ or $\mathbb{H}^3_1$. 
If the mean and Gaussian curvatures are constant and the shape operator is not diagonalizable at a point, 
then $M^2_1$ is an open part of a complex circle or a B-scroll.}
\label{KimKim}
\end{fact}

B. Bekta\c{s}, E.\"O. Canfes and U. Dursun (\cite{BCD1}) determined the type number of the pseudo-spherical Gauss 
map of an oriented Lorentzian surface in $\mathbb{S}^3_1$ of non-zero constant mean curvature and non-diagonalizable 
shape operator at a point.  

\begin{fact}
{\sl An oriented Lorentzian surface in $\mathbb{S}^3_1$ of constant mean curvature and non-diagonalizable shape 
operator is of null 2-type pseudo-spherical Gauss map if and only if it is an open part of a non-flat B-scroll over 
a null curve.}
\label{BCD}
\end{fact}
In this paper, we determined the type numbers of the pseudo-hyperbolic Gauss maps of such Lorentzian surfaces in 
$\mathbb{H}^3_1$. 
\begin{thm}
{\sl Let $M$ be an oriented Lorentzian surface in $\mathbb{H}^3_1$ of %
constant mean and Gaussian curvatures and non-diagonalizable shape operator.  
The following facts hold. \\
\begin{tabular}{cl}
(\rnum{1}) &If $M$ is the complex circle of radius $-1$, then the pseudo-hyperbolic Gauss map is of $1$-type.\\
(\rnum{2}) &If $M$ is the complex circle of radius $\kappa$ ($\mathrm{Re}(\kappa)=-1$, $\kappa \neq -1$), 
then the pseudo-hyperbolic\\
 & Gauss map is of infinite type. \\
(\rnum{3}) &If $M$ is a non-flat B-scroll, then the pseudo-hyperbolic Gauss is of null $2$-type. \\
(\rnum{4}) &If $M$ is a flat B-scroll, then the pseudo-hyperbolic Gauss map is of infinite type. 
\end{tabular}
}\label{thmA}
\end{thm}

\begin{remark} The shapes of the complex circles of radius $-1$ and $\kappa$ 
($\mathrm{Re}(\kappa)=-1$, $\kappa \neq -1$) are as in Figure 1.  
\end{remark}

\begin{thm}
{\sl 
Let $M$ be a complex circle in $\mathbb{H}^3_1$ and $u$ be any real number. 
The parallel surface $M^u$ of $M$ at distance $u$ is a complex circle and the radius $\kappa^u$ of 
the complex circle $M^u$ moves over the whole of $\{ z\in \mathbb{C} \,|\, \mathrm{Re}(z)=-1 \}$
when $u$ moves over $\mathbb{R}$. 
Hence the only parallel surface of $M$ has the pseudo-hyperbolic Gauss map of $1$-type and other parallel surfaces 
of $M$ have the pseudo-hyperbolic Gauss map of infinite type.
}
\label{thmD}
\end{thm}

\begin{thm}
{\sl
Let $M$ be a B-scroll in $\mathbb{H}^3_1$ and $u\in \mathbb{R}$ sufficiently close to $0$. 
If $M$ is flat (resp. non-flat), then the parallel surface $M_u$ also is flat (resp. non-flat) B-scroll. 
Hence the type numbers of the pseudo-hyperbolic Gauss maps of the parallel surfaces of a B-scroll are equal to 
that of the original B-scroll. 
} 
\label{thmB}
\end{thm}
Furthermore, we prove the following fact for the pseudo-hyperbolic Gauss map of one of the $n$-dimensional 
Lorentzian hypersurfaces of B-scroll-type given by of L.J. Alias, A. Ferrandez, and P. Lucas(\cite{AFL}). \\
\begin{thm}
{\sl Let $(A,B,C,Z_1,\cdots ,Z_{n-2})$ be the Cartan frame field of a null curve $\gamma(s) \subset \mathbb{H}^{n+1}_1$. 
Assume that $k_1(s)\neq 0$, $k_2^2=1$ and $k_i$ is non-zero constant for $i=3,4,\cdots ,n-2$. 
The immersion ${\bf x}:I\times \mathbb{R}\times \mathbb{R}^{n-2} \rightarrow \mathbb{H}^{n+1}_1\subset \mathbb{E}^{n+2}_2$ given by
$${\bf x}(s,t,z)=\left(1+\frac{|z|^2}{2}\right)\gamma(s)+tB(s)+\sum_{j=1}^{n-2}z_jZ_j(s)-\frac{k_2|z|^2}{2}C(s),$$
parametrizes an oriented Lorentzian hypersurface of $\mathbb{H}^{n+1}_1$. 
The pseudo-hyperbolic Gauss map of this oriented Lorentzian hypersurface is of infinite type.}
\label{thmC}
\end{thm}

\newpage

$\,$

\vspace{0.5truecm}

\centerline{
$\displaystyle{
}$}

\vspace{0.1truecm}


$$\left(\begin{tabular}{l}
$D:=\pi(S^1_{\mathbb{C}}(\kappa))$ 
($\pi\,:\,$ the natural projection of $\mathbb{R}^4_2=\mathbb{R}^2_0\times \mathbb{R}^2_2$ onto $\mathbb{R}^2_0$)\\
$\displaystyle{L_{\kappa}:=
\begin{cases}
\pi^{-1}(\partial D)\cap S^1_{\mathbb{C}}(\kappa) & (\mathrm{Im}\kappa \neq 0) \\
\pi^{-1}((0,0))\cap S^1_{\mathbb{C}}(\kappa) & (\mathrm{Im}\kappa =0)
\end{cases}}$\\
$l_i :=\pi^{-1}(\gamma_i)\cap S^1_{\mathbb{C}}(\kappa) \quad (i=1,2)$\\
\end{tabular}\right)$$

\vspace{0.25truecm}

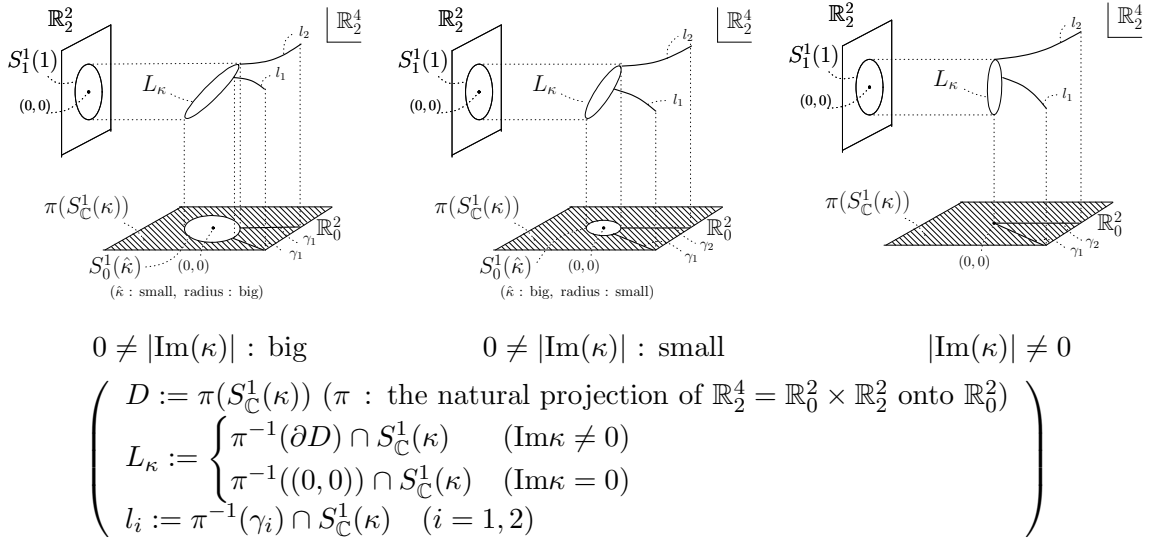
\centerline{Figure 1 : The shape of $S^1_{\mathbb{C}}(\kappa)$ in $\mathbb{R}^4_2$}

\vspace{0.5truecm}



\section{Basic notions and facts}

Let $\langle\,\,,\,\,\rangle$  be the non-degenerate symmetric linear form of the $m$-dimensional real vector space $\mathbb{R}^m$ defined by
$$\langle v,w\rangle =-\sum_{i=1}^{s}v_iw_i+\sum_{j=s+1}^{m}v_jw_j,$$
where $v=(v_1,v_2,\cdots ,v_m)$ and $w=(w_1,w_2,\cdots ,w_m)$.
Let  $\tilde{g}$ be the pseudo-Euclidean metric of index $s$ on the $m$-dimensional affine space $\mathbb{R}^m$ induced from $\langle , \rangle$ , i.e.
$$\tilde{g}=-\sum_{i=1}^{s}dx_i^2+\sum_{j=s+1}^{m}dx_j^2,$$
where $(x_1,\cdots ,x_m)$ denotes an affine coordinate of $\mathbb{R}^m$. 
Denote $(\mathbb{R}^m,\tilde{g})$ by $\mathbb{E}^m_s$.
For $x_0\in \mathbb{E}^m_s$ and $c>0$ , we put\\
$$\mathbb{S}^{m-1}_{s}(x_0,c)=\{ x=(x_1,\cdots ,x_m)\in \mathbb{E}^m_s \, \vert \, \langle \overrightarrow{x_0x},\overrightarrow{x_0x}\rangle=\frac{1}{c}\, \}$$
and
$$\mathbb{H}^{m-1}_{s-1}(x_0,-c)=\{ x=(x_1,\cdots ,x_m)\in \mathbb{E}^m_s\, \vert \, \langle \overrightarrow{x_0x},\overrightarrow{x_0x}\rangle=-\frac{1}{c}\, \}$$
Then $\mathbb{S}^{m-1}_{s}(x_0,c)$ (resp. $\mathbb{H}^{m-1}_{s-1}(x_0,-c)$) is the $(m-1)$-dimensional pseudo-Riemannian submanifold in $\mathbb{E}^m_s$ 
of constant curvature $c$ and index $s$ (resp. constant curvature $-c$ and index $s-1$), 
called {\it a pseudo-sphere} (resp. {\it a pseudo-hyperbolic space}). 
In particular, $\mathbb{S}^{m-1}_1(c)$ (resp. $\mathbb{H}^{m-1}_1(-c)$) is called {\it a de Sitter space} (resp. {\it an anti-de Sitter space}). 
For the simplicity, denote $\mathbb{S}^{m-1}_{s}(0,c)$ and $\mathbb{H}^{m-1}_{s-1}(0,-c)$ by $\mathbb{S}^{m-1}_{s}(c)$ and $\mathbb{H}^{m-1}_{s-1}(-c)$, 
where $0$ is the origin of $\mathbb{E}^m_s$. 
Furthermore, we abbreviate $\mathbb{S}^{m-1}_s(1)$ (resp. $\mathbb{H}^{m-1}_{s-1}(-1)$) by $\mathbb{S}^{m-1}_s$ (resp. $\mathbb{H}^{m-1}_{s-1}$).

Let $M_t$ be an $n$-dimensional pseudo-Riemannian submanifold of index $t$ in $\mathbb{E}^m_s$ immersed by ${\bf x}$, for the simplicity, we denote ${\bf x}(M_t)$ by $M_t$. 
Let $\nabla$ and $\tilde{\nabla}$ be the Levi-Civita connections of $M_t$ and $\mathbb{E}^m_s$, respectively. 
Also, let $\tilde{\nabla}^{\bf x}$ the induced connection of $\tilde{\nabla}$ by ${\bf x}$ and $\nabla^\bot$ the normal connection of $M_t$.\\
For the simplicity, we denote all metrics by the common symbol $\langle\,\,,\,\,\rangle$. 
We take a local orthonormal frame field $(e_1,\cdots ,e_n)$ of the tangent bundle $TM_t$ of $M_t$ defined on an open set $U$ of $M_t$ and 
a local orthonormal frame field $(e_{n+1},\cdots ,e_m)$ of the normal bundle $T^\bot M_t$ of $M_t$ defined on $U$. 
Put $\epsilon_A:=\langle e_A,e_A\rangle=\pm 1$ ($A=1,\ldots ,m$). 
Let $\{ \tilde{\omega}_{AB}\}_{A,B=1,\cdots ,m}$ be the connection form of $\tilde{\nabla}$ with respect to $(e_1, \cdots ,e_n,e_{n+1},\cdots ,e_m)$, 
$\{ \omega_{ij}\}_{i,j=1,\cdots ,n}$ the connection form of $\nabla$ with respect to $(e_1,\cdots ,e_n)$ 
and $\{ \omega_{rs}^\bot \}_{r,s=n+1,\cdots ,m}$ the connection form of $\nabla^\bot$ with respect to $(e_{n+1},\cdots ,e_m)$, that is, 
$$\tilde{\nabla}_{X}e_A=\sum_{B=1}^m\epsilon_B\tilde{\omega}_{AB}(X)e_B$$
$$\nabla_{X}e_i=\sum_{j=1}^n\epsilon_j\omega_{ij}(X)e_j$$
and
$$\nabla^\bot_{X}e_r=\sum_{s=n+1}^m\epsilon_s\omega^\bot_{rs}(X)e_s$$
for $X\in TM_t$. 
Also, let $h^r_{ij}$ be the components of the second fundamental form $h$ of $M_t$ with respect to $(e_1,\cdots ,e_n,e_{n+1},\cdots ,e_m)$, 
that is, $h(e_i,e_j)=\sum_{r=n+1}^{m}h_{ij}^{r}e_r$ and $A_r$ the shape operator of $M_t$ for $e_r$, 
that is, $A_r=A_{e_r}$, where $A$ denotes the shape tensor of $M_t$. Then we have $\omega_{AB}+\omega_{BA}=0$,
\begin{equation}
\tilde{\nabla} _{e_k} e_i=\sum_{j=1}^{n} \epsilon_j \omega_{ij}(e_k)e_j + \sum_{r=n+1}^m \epsilon_r h_{ki}^r e_r \\
\end{equation}
and
\begin{equation}
\tilde{\nabla} _{e_k} e_r=-A_r(e_k)+\sum_{s=n+1}^{m} \epsilon_s \omega^\bot_{rs}(e_k) e_s
\end{equation}
The mean curvature vector $H$ and the scalar curvature $S$ are defined by;
\begin{equation}
H=\frac{1}{n}\sum_{r=n+1}^m \epsilon_r\mathrm{tr}A_r e_r
\end{equation}
and
\begin{equation}
S=n^2\langle H,H\rangle-\Vert h\Vert^2, \label{eq8}
\end{equation}
where $\Vert h\Vert^2=\sum_{i,j=1}^{n}\sum_{r=n+1}^{m} \epsilon_i\epsilon_j\epsilon_r h_{ij}^r h_{ji}^r$. 
Denote by $\hat{\nabla}h$ the covariant derivative of $h$ with respect to $\nabla$ and $\nabla^\bot$. 
Let 
\begin{eqnarray}
(\hat{\nabla}_{e_k}h)(e_i,e_j)=\sum_{r=n+1}^m\epsilon_r h_{ij;k}^re_r.
\end{eqnarray}
Then we have
\begin{equation}
h_{ij;k}^{r}=h_{jk;i}^r,
\end{equation}
\begin{equation}
h_{jk;i}^{r}=e_i(h_{jk}^r)-\sum_{l=1}^n \epsilon_l (h_{lk}^r \omega_{jl}(e_i)+h_{lj}^r \omega_{kl}(e_i))+\sum_{s=n+1}^m \epsilon_s h_{jk}^s \omega^\bot_{sr}(e_i),
\end{equation}
\begin{equation}
R^\bot(e_j,e_k;e_r,e_s)=\langle [A_r,A_s](e_j),e_k\rangle =\sum_{i=1}^n \epsilon_i(h_{ik}^r h_{ij}^s-h_{ij}^r h_{ik}^s),
\end{equation}
where $R^\bot$ is the normal curvature tensor of $M_t$.

Let ${\bf x}:M_t\hookrightarrow \mathbb{S}^{m-1}_{s}(c)$ or $\mathbb{H}^{m-1}_{s-1}(-c)\subset \mathbb{E}^m_{s}$ be an isometric immersion. 
Denote by $h$ and $H$ the second fundamental form and the mean curvature vector of $M_t$ in $\mathbb{E}^m_s$, 
$\hat{h}$ and $\hat{H}$ be the second fundamental form and of the mean curvature vector of $M_t$ in $\mathbb{S}^{m-1}_s(c)$ or $\mathbb{H}^{m-1}_{s-1}(-c)$.\\
We have
\begin{eqnarray}
H&=&\hat {H}-\epsilon c{\bf x} ,\\ 
h&=&\hat {h}(X,Y)-\epsilon c\langle X,Y\rangle {\bf x} ,
\end{eqnarray}
and (\ref{eq8}) is rewritten as
\begin{eqnarray}
S=\epsilon cn(n-1)+n^2\langle \hat {H},\hat {H}\rangle -\Vert \hat {h} \Vert^2 ,\label{eq5}
\end{eqnarray}
where $\epsilon=+1$ if in $\mathbb{S}^{m-1}_{s}(c)$ and $\epsilon=-1$ if in $\mathbb{H}^{m-1}_{s-1}(-c)$.\\
The gradient vector field $\nabla f$ of $f\in C^\infty (M_t)$ is defined by $\nabla f=\sum_{i=1}^n \epsilon_i e_i(f)e_i$, 
and Laplacian operator $\Delta$ of $M_t$ with respect to the induced metric is given by $\Delta =\sum_{i=1}^n \epsilon_i(\nabla _{e_i}e_i -e_ie_i)$.

\begin{defi}
{\rm Let $\phi : M_t\rightarrow \mathbb{H}^{m-1}_{s-1}(-c)\subset \mathbb{E}^m_s$ (resp. $\phi : M_t\rightarrow \mathbb{S}^{m-1}_s(c)\subset \mathbb{E}^m_s$) 
be a smooth map. 
Then $\phi$ is said to be {\it of finite type} in $\mathbb{H}^{m-1}_{s-1}$ (resp. in $\mathbb{S}^{m-1}_{s}$) 
if $\phi$ has the following spectral decomposition as follows;
$$\phi =\phi_1+\phi_2+\cdots +\phi_k$$
where $\phi_i : M_t\rightarrow \mathbb{E}^m_s$'s are non-constant map such that 
$\Delta \phi_i=\lambda_i\phi_i$ with $\lambda_i\in \mathbb{R}$, $i=1,2,\ldots ,k$.
If $\phi$ has this spectral decomposition and $\lambda_i$'s are mutually distinct constant, then the map $\phi$ is said to be {\it of $k$-type},
and  when one of $\lambda_i$'s is equal to zero, the map $\phi$ is said to be {\it of null $k$-type.}}
\end{defi}
For a map of finite type, the following fact holds.
\begin{lemm}
{\rm Let $\phi : M_t\rightarrow \mathbb{H}^{m-1}_{s-1}$ or $\mathbb{S}^{m-1}_s$ be a smooth map. 
If $\Delta ^2\phi=0$, then $\Delta \phi=0$ or $\phi$ is of infinite type.} \label{lemm1}
\end{lemm}
\begin{proof} 
Assume that $\phi$ is of finite type and $\phi$ has the following spectral decomposition; 
$\phi=\phi_1+\phi_2+\cdots +\phi_k$ with $\Delta \phi_i=\lambda_i\phi_i$ for $\lambda_i \in \mathbb{R}$ and $i=1,2,\cdots ,k$. 
Then we have 
$$0=\Delta ^2\phi=\lambda_1^2\phi_1+\cdots +\lambda_k^2\phi_k.$$
Therefore we have $k=1$ and $\lambda_1=0$, that is, $\Delta \phi=0$.
\end{proof}

Let $\mathbb{C}^{n+1}$ be the $(n+1)$-dimensional complex vector space which is identified with $\mathbb{R}^{2n+2}$. 
Define a non-degenerate symmetric bilinear form $\langle\,\,,\,\,\rangle$ of the $\mathbb{C}^{n+1} (=\mathbb{R}^{2n+2})$ by 
\begin{equation}
\hspace{30pt} \langle z, w\rangle=\mathrm{Re} \left(\sum_{i=1}^{n+1}z_iw_i\right) \quad (z=(z_1,\cdots ,z_{n+1}), w=(w_1, \cdots ,w_{n+1}) \in \mathbb{C})
\end{equation} 
Note that $\langle z,z\rangle =\sum_{i=1}^{n+1} x_i^2-\sum_{i=1}^{n+1} y_i^2$ when $z=(x_1+\sqrt{-1}y_1, \cdots , x_{n+1}+\sqrt{-1}y_{n+1})$ 
($x_i, y_i \in \mathbb{R}$ $i=1,\cdots ,n+1$)). 
Let $\tilde{g}$ be a pseudo-Euclidean metric of index $n+1$ on the $(2n+2)$-dimensional affine space $\mathbb{R}^{2n+2} (=\mathbb{C}^{n+1})$ induced from $\langle\,\,,\,\,\rangle$. 

\begin{defi}
{\rm Fix a non-zero complex number number $\kappa$.
We put
$$S^n_{\mathbb{C}}(\kappa)=\{ (z_1, z_2, \cdots ,z_{n+1}) \in \mathbb{C}^{n+1}=\mathbb{E}^{2n+2}_{n+1}\, \vert \, \sum_{i=1}^{n+1}z_i^2=\kappa \} ,$$
that is, $\sum_{i=1}^{n+1}(x_i^2-y_i^2)=\mathrm{Re}(\kappa)$ and $2\sum_{i=1}^{n+1}x_iy_i=\mathrm{Im}(\kappa)$ for $z_i=x_i+\sqrt{-1}y_i$ ($x_i, y_i \in \mathbb{R}$). 
This submanifold $S^n_{\mathbb{C}}(\kappa)$ is called {\it a complex sphere of radius $\kappa$}. 
In particular, when $n=1$, it is {\it a complex circle of radius $\kappa$}. 
}
\end{defi}

\begin{remark}
$$S^n_{\mathbb{C}}(\kappa) \subset \left \{
\begin{array}{ll}
\mathbb{S}^{2n+1}_{n+1}\left( \mathrm{Re}(\kappa) \right) & (\mbox{if} \quad \mathrm{Re}(\kappa)>0)\\
& \\
\mathbb{H}^{2n+1}_{n}\left( \mathrm{Re}(\kappa) \right) & (\mbox{if} \quad \mathrm{Re}(\kappa)<0)
\end{array}
\right.
$$
\end{remark}
The complex circle $S^1_{\mathbb{C}}(\kappa) \subset \mathbb{H}^3_1$ is parameterized as 
${\bf x}(z)=\sqrt{\kappa}(\cos z,\sin,z)$ \,($z\in \mathbb{C}$).

\begin{defi}
{\rm
A tangent vector $v$ of pseudo-Riemannian manifold $M$ is said to be}
$$\begin{array}{lll}
{\it spacelike} & {\rm if}& \langle v,v\rangle >0 \quad {\rm or}\quad  v=0, \\
{\it timelike} &{\rm if}& \langle v,v\rangle <0, \\
{\it null} &{\rm if}& \langle v,v\rangle =0 \quad {\rm and}\quad v\neq 0.
\end{array}$$
{\rm
The category into which a given tangent vector falls is called its {\it causal character}. The causal character of a curve $\gamma$ in $M$ is that of the velocity $\gamma'$.}
\end{defi}

\begin{defi}
{\rm Let $\gamma(s)$ be a null curve in $\mathbb{S}^3_1 (\subset \mathbb{E}^4_{1})$ or $\mathbb{H}^3_1 (\subset \mathbb{E}^4_2)$. 
Let $(A,B,C)$ be a tangent frame field of $\mathbb{S}^3_1$ or $\mathbb{H}^3_1$ along $\gamma$ such that;
\begin{eqnarray}
\begin{array}{l}
\dot{\gamma}(s)=A(s),\\
\dot{A}(s)= \tilde{\nabla}_{\dot{\gamma}} A(s)+\langle A(s),A(s)\rangle \hat{H}_\gamma=k_1(s)C(s),  \\
\dot{C}(s)= \tilde{\nabla}_{\dot{\gamma}} C(s)+\langle A(s),C(s)\rangle \hat{H}_\gamma=k_2A(s)+k_1(s)B(s), \nonumber \\
\dot{B}(s)= \tilde{\nabla}_{\dot{\gamma}} B(s)+\langle A(s),B(s)\rangle \hat{H}_\gamma=k_2C(s)+\epsilon \gamma, \nonumber 
\end{array}
\end{eqnarray}
where $\epsilon=+1$ if in $\mathbb{S}^{3}_{1}$, $\epsilon=-1$ if in $\mathbb{H}^{3}_{1}$, $k_1$ and $k_2$ are 
positive-valued functions.  This frame field $(A,B,C)$ is called the {\it Cartan frame field along} $\gamma$.  
The the function $k_1$ and $k_2$ are called {\it the first curvature} and {\it the second curvature} of $\gamma$, 
respectively.  
Then the immersion ${\bf x}(s,t)=\gamma(s)+tB(s)$ parametrizes a Lorentzian ruled surface $M$ in $\mathbb{S}^3_1$ 
or $\mathbb{H}^3_1$.
In this paper, if $k_2$ is constant, then $M$ is called {\it the B-scroll over} $\gamma$.}
\label{defi1}
\end{defi}

\section{Pseudo-spherical and Pseudo-hyperbolic Gauss map}
Let $G(n+1,m)$ be Grassmannian manifold 
consisting of $(n+1)$-dimensional oriented non-degenerate subspaces of $\mathbb{E}^m_s$, 
Let $G(n+1,m)_{t+1}$ be the submanifold of $G(n+1,m)$ consisting of $(n+1)$-dimensional oriented non-degenerate subspaces of index $t+1$ of $\mathbb{E}^m_s$. 
Let $(\widetilde{e}_1,\cdots ,\widetilde{e}_m)$ and $(\widehat{e}_1,\cdots ,\widehat{e}_m)$ be two orthonormal frames of $\mathbb{E}^m_s$. 
Let $\widetilde{e}_{i_1}\wedge \cdots \wedge \widetilde{e}_{i_{n+1}}$ and $\widehat{e}_{j_1}\wedge \cdots \wedge \widehat{e}_{j_{n+1}}$ be two vectors in $\bigwedge^{n+1}\mathbb{E}^m_s$. 
Define an indefinite inner product $\langle \langle , \rangle \rangle$ on $\bigwedge^{n+1}\mathbb{E}^m_s$ by 
$$\langle \langle \widetilde{e}_{i_1}\wedge \cdots \wedge \widetilde{e}_{i_{n+1}} \,,\, \widehat{e}_{j_1}\wedge \cdots \wedge \widehat{e}_{j_{n+1}}\rangle \rangle 
=\mathrm{det} \left( \langle \widetilde{e}_{i_l} , \widehat{e}_{j_k}\rangle \right),\quad (l,k=1,\cdots ,n+1).$$

\newpage


\centerline{
\scriptsize{
\unitlength 0.1in
\begin{picture}( 67.7000, 43.2000)( -5.9000,-45.3000)
%
\special{pn 8}%
\special{pa 3012 4516}%
\special{pa 2988 4490}%
\special{pa 2964 4464}%
\special{pa 2940 4438}%
\special{pa 2916 4410}%
\special{pa 2894 4384}%
\special{pa 2870 4358}%
\special{pa 2846 4332}%
\special{pa 2824 4306}%
\special{pa 2800 4278}%
\special{pa 2778 4252}%
\special{pa 2754 4226}%
\special{pa 2732 4200}%
\special{pa 2710 4172}%
\special{pa 2688 4146}%
\special{pa 2668 4120}%
\special{pa 2646 4092}%
\special{pa 2626 4066}%
\special{pa 2604 4038}%
\special{pa 2584 4012}%
\special{pa 2564 3984}%
\special{pa 2546 3958}%
\special{pa 2526 3930}%
\special{pa 2508 3904}%
\special{pa 2490 3876}%
\special{pa 2472 3848}%
\special{pa 2456 3822}%
\special{pa 2440 3794}%
\special{pa 2424 3766}%
\special{pa 2408 3738}%
\special{pa 2394 3710}%
\special{pa 2380 3682}%
\special{pa 2366 3654}%
\special{pa 2354 3626}%
\special{pa 2342 3598}%
\special{pa 2330 3570}%
\special{pa 2320 3542}%
\special{pa 2310 3514}%
\special{pa 2300 3484}%
\special{pa 2292 3456}%
\special{pa 2284 3428}%
\special{pa 2278 3398}%
\special{pa 2272 3370}%
\special{pa 2268 3340}%
\special{pa 2264 3310}%
\special{pa 2260 3280}%
\special{pa 2258 3252}%
\special{pa 2258 3222}%
\special{pa 2258 3192}%
\special{pa 2258 3162}%
\special{pa 2260 3132}%
\special{pa 2262 3100}%
\special{pa 2264 3070}%
\special{pa 2268 3040}%
\special{pa 2274 3008}%
\special{pa 2278 2978}%
\special{pa 2284 2948}%
\special{pa 2290 2916}%
\special{pa 2298 2884}%
\special{pa 2306 2854}%
\special{pa 2314 2822}%
\special{pa 2322 2792}%
\special{pa 2332 2760}%
\special{pa 2340 2728}%
\special{pa 2350 2696}%
\special{pa 2360 2664}%
\special{pa 2370 2634}%
\special{pa 2380 2602}%
\special{pa 2392 2570}%
\special{pa 2402 2538}%
\special{pa 2412 2506}%
\special{pa 2424 2474}%
\special{pa 2434 2442}%
\special{pa 2446 2410}%
\special{pa 2456 2378}%
\special{pa 2468 2346}%
\special{pa 2478 2314}%
\special{pa 2490 2284}%
\special{pa 2500 2252}%
\special{pa 2510 2220}%
\special{pa 2520 2188}%
\special{pa 2530 2156}%
\special{pa 2540 2124}%
\special{pa 2548 2092}%
\special{pa 2556 2060}%
\special{pa 2566 2028}%
\special{pa 2574 1998}%
\special{pa 2580 1966}%
\special{pa 2586 1934}%
\special{pa 2594 1902}%
\special{pa 2598 1872}%
\special{pa 2604 1840}%
\special{pa 2608 1808}%
\special{pa 2612 1778}%
\special{pa 2614 1746}%
\special{pa 2616 1716}%
\special{pa 2616 1686}%
\special{pa 2616 1654}%
\special{pa 2616 1624}%
\special{pa 2614 1594}%
\special{pa 2612 1564}%
\special{pa 2608 1532}%
\special{pa 2604 1502}%
\special{pa 2598 1472}%
\special{pa 2592 1442}%
\special{pa 2584 1414}%
\special{pa 2576 1384}%
\special{pa 2568 1354}%
\special{pa 2558 1324}%
\special{pa 2548 1294}%
\special{pa 2536 1266}%
\special{pa 2526 1236}%
\special{pa 2514 1208}%
\special{pa 2500 1178}%
\special{pa 2486 1150}%
\special{pa 2472 1120}%
\special{pa 2458 1092}%
\special{pa 2444 1062}%
\special{pa 2428 1034}%
\special{pa 2412 1006}%
\special{pa 2396 976}%
\special{pa 2378 948}%
\special{pa 2362 920}%
\special{pa 2344 890}%
\special{pa 2326 862}%
\special{pa 2308 834}%
\special{pa 2290 806}%
\special{pa 2272 778}%
\special{pa 2252 748}%
\special{pa 2234 720}%
\special{pa 2214 692}%
\special{pa 2196 664}%
\special{pa 2176 636}%
\special{pa 2156 608}%
\special{pa 2140 584}%
\special{sp}%
%
\special{pn 8}%
\special{pa 1790 4138}%
\special{pa 2780 2410}%
\special{fp}%
\put(26.1000,-29.1000){\makebox(0,0)[lb]{$B(s_2)$}}%
\put(22.7000,-41.6000){\makebox(0,0)[rt]{$\dot{\gamma}(s_1)=A(s_1)$}}%
\put(26.8000,-7.8000){\makebox(0,0)[lt]{$B(s_3)$}}%
%
\special{pn 13}%
\special{pa 2560 3960}%
\special{pa 2820 3570}%
\special{fp}%
\special{sh 1}%
\special{pa 2820 3570}%
\special{pa 2766 3614}%
\special{pa 2790 3614}%
\special{pa 2800 3638}%
\special{pa 2820 3570}%
\special{fp}%
%
\special{pn 13}%
\special{pa 2550 3980}%
\special{pa 2350 3690}%
\special{fp}%
\special{sh 1}%
\special{pa 2350 3690}%
\special{pa 2372 3756}%
\special{pa 2380 3734}%
\special{pa 2404 3734}%
\special{pa 2350 3690}%
\special{fp}%
%
\special{pn 4}%
\special{ar 2550 3610 320 80  0.0000000 6.2831853}%
%
\special{pn 4}%
\special{pa 2240 3630}%
\special{pa 2540 3960}%
\special{fp}%
%
\special{pn 4}%
\special{ar 2290 2950 320 80  0.0000000 6.2831853}%
%
\special{pn 4}%
\special{pa 2620 2950}%
\special{pa 2290 3290}%
\special{fp}%
%
\special{pn 4}%
\special{pa 1982 2982}%
\special{pa 2270 3322}%
\special{fp}%
%
\special{pn 4}%
\special{ar 2410 630 300 50  0.0000000 6.2831853}%
%
\special{pn 4}%
\special{pa 2120 630}%
\special{pa 2390 990}%
\special{fp}%
%
\special{pn 4}%
\special{pa 2880 3630}%
\special{pa 2540 3980}%
\special{fp}%
\put(17.2000,-31.7000){\makebox(0,0)[rt]{light cone}}%
%
\special{pn 8}%
\special{pa 2100 3090}%
\special{pa 2082 3116}%
\special{pa 2060 3140}%
\special{pa 2038 3162}%
\special{pa 2012 3182}%
\special{pa 1984 3200}%
\special{pa 1954 3212}%
\special{pa 1924 3222}%
\special{pa 1892 3228}%
\special{pa 1860 3230}%
\special{pa 1828 3230}%
\special{pa 1798 3228}%
\special{pa 1766 3222}%
\special{pa 1750 3220}%
\special{sp -0.045}%
%
\special{pn 8}%
\special{pa 3012 4516}%
\special{pa 2988 4490}%
\special{pa 2964 4464}%
\special{pa 2940 4438}%
\special{pa 2916 4410}%
\special{pa 2894 4384}%
\special{pa 2870 4358}%
\special{pa 2846 4332}%
\special{pa 2824 4304}%
\special{pa 2800 4278}%
\special{pa 2778 4252}%
\special{pa 2754 4226}%
\special{pa 2732 4198}%
\special{pa 2710 4172}%
\special{pa 2688 4146}%
\special{pa 2666 4118}%
\special{pa 2646 4092}%
\special{pa 2624 4064}%
\special{pa 2604 4038}%
\special{pa 2584 4010}%
\special{pa 2564 3984}%
\special{pa 2544 3956}%
\special{pa 2526 3930}%
\special{pa 2508 3902}%
\special{pa 2490 3876}%
\special{pa 2472 3848}%
\special{pa 2454 3820}%
\special{pa 2438 3792}%
\special{pa 2422 3766}%
\special{pa 2408 3738}%
\special{pa 2392 3710}%
\special{pa 2378 3682}%
\special{pa 2366 3654}%
\special{pa 2352 3626}%
\special{pa 2340 3598}%
\special{pa 2330 3570}%
\special{pa 2318 3542}%
\special{pa 2308 3512}%
\special{pa 2300 3484}%
\special{pa 2292 3456}%
\special{pa 2284 3426}%
\special{pa 2278 3398}%
\special{pa 2272 3368}%
\special{pa 2268 3340}%
\special{pa 2264 3310}%
\special{pa 2260 3280}%
\special{pa 2258 3252}%
\special{pa 2258 3222}%
\special{pa 2258 3192}%
\special{pa 2258 3162}%
\special{pa 2260 3132}%
\special{pa 2262 3102}%
\special{pa 2266 3070}%
\special{pa 2270 3040}%
\special{pa 2274 3010}%
\special{pa 2280 2980}%
\special{pa 2286 2948}%
\special{pa 2292 2918}%
\special{pa 2300 2886}%
\special{pa 2308 2856}%
\special{pa 2316 2824}%
\special{pa 2324 2794}%
\special{pa 2332 2762}%
\special{pa 2342 2730}%
\special{pa 2352 2698}%
\special{pa 2362 2668}%
\special{pa 2372 2636}%
\special{pa 2382 2604}%
\special{pa 2394 2572}%
\special{pa 2404 2542}%
\special{pa 2416 2510}%
\special{pa 2426 2478}%
\special{pa 2438 2446}%
\special{pa 2448 2414}%
\special{pa 2460 2382}%
\special{pa 2470 2350}%
\special{pa 2482 2318}%
\special{pa 2492 2286}%
\special{pa 2502 2254}%
\special{pa 2514 2222}%
\special{pa 2524 2192}%
\special{pa 2532 2160}%
\special{pa 2542 2128}%
\special{pa 2552 2096}%
\special{pa 2560 2064}%
\special{pa 2568 2032}%
\special{pa 2576 2000}%
\special{pa 2582 1968}%
\special{pa 2590 1938}%
\special{pa 2596 1906}%
\special{pa 2600 1874}%
\special{pa 2606 1842}%
\special{pa 2610 1812}%
\special{pa 2612 1780}%
\special{pa 2616 1748}%
\special{pa 2616 1718}%
\special{pa 2618 1686}%
\special{pa 2618 1656}%
\special{pa 2616 1624}%
\special{pa 2614 1594}%
\special{pa 2612 1562}%
\special{pa 2608 1532}%
\special{pa 2602 1502}%
\special{pa 2596 1470}%
\special{pa 2590 1440}%
\special{pa 2582 1410}%
\special{pa 2574 1380}%
\special{pa 2564 1350}%
\special{pa 2554 1320}%
\special{pa 2544 1290}%
\special{pa 2532 1260}%
\special{pa 2520 1232}%
\special{pa 2508 1202}%
\special{pa 2494 1172}%
\special{pa 2482 1142}%
\special{pa 2466 1114}%
\special{pa 2452 1084}%
\special{pa 2436 1056}%
\special{pa 2422 1026}%
\special{pa 2406 998}%
\special{pa 2388 970}%
\special{pa 2372 940}%
\special{pa 2354 912}%
\special{pa 2338 884}%
\special{pa 2320 856}%
\special{pa 2302 828}%
\special{pa 2284 800}%
\special{pa 2266 772}%
\special{pa 2248 744}%
\special{pa 2230 716}%
\special{pa 2212 688}%
\special{pa 2192 662}%
\special{pa 2174 634}%
\special{pa 2156 606}%
\special{pa 2138 580}%
\special{pa 2120 552}%
\special{pa 2102 526}%
\special{pa 2084 498}%
\special{pa 2066 472}%
\special{pa 2048 446}%
\special{pa 2030 420}%
\special{pa 2012 392}%
\special{pa 1994 366}%
\special{pa 1976 340}%
\special{pa 1960 314}%
\special{pa 1942 288}%
\special{pa 1930 270}%
\special{sp}%
%
\special{pn 13}%
\special{pa 2262 3290}%
\special{pa 2220 3050}%
\special{fp}%
\special{sh 1}%
\special{pa 2220 3050}%
\special{pa 2212 3120}%
\special{pa 2230 3104}%
\special{pa 2252 3112}%
\special{pa 2220 3050}%
\special{fp}%
%
\special{pn 4}%
\special{pa 1790 4138}%
\special{pa 2780 2410}%
\special{fp}%
%
\special{pn 13}%
\special{pa 2270 3302}%
\special{pa 2520 2900}%
\special{fp}%
\special{sh 1}%
\special{pa 2520 2900}%
\special{pa 2468 2946}%
\special{pa 2492 2946}%
\special{pa 2502 2968}%
\special{pa 2520 2900}%
\special{fp}%
\put(30.7000,-45.3000){\makebox(0,0)[lt]{$\gamma$}}%
\put(29.4000,-36.3000){\makebox(0,0)[lb]{$B(s_1)$}}%
\put(21.8000,-7.7000){\makebox(0,0)[rt]{$\dot{\gamma}(s_3)=A(s_3)$}}%
%
\special{pn 13}%
\special{pa 2550 3980}%
\special{pa 2350 3690}%
\special{fp}%
\special{sh 1}%
\special{pa 2350 3690}%
\special{pa 2372 3756}%
\special{pa 2380 3734}%
\special{pa 2404 3734}%
\special{pa 2350 3690}%
\special{fp}%
%
\special{pn 4}%
\special{ar 2550 3610 320 80  0.0000000 6.2831853}%
%
\special{pn 4}%
\special{pa 2240 3630}%
\special{pa 2540 3960}%
\special{fp}%
%
\special{pn 8}%
\special{pa 6180 1162}%
\special{pa 4770 2842}%
\special{fp}%
%
\special{pn 4}%
\special{pa 5356 1110}%
\special{pa 6156 1828}%
\special{fp}%
%
\special{pn 4}%
\special{pa 5058 1460}%
\special{pa 5860 2178}%
\special{fp}%
%
\special{pn 4}%
\special{pa 4778 1836}%
\special{pa 5578 2554}%
\special{fp}%
%
\special{pn 4}%
\special{pa 4480 2168}%
\special{pa 5282 2886}%
\special{fp}%
%
\special{pn 13}%
\special{pa 5860 1566}%
\special{pa 5636 1364}%
\special{fp}%
\special{sh 1}%
\special{pa 5636 1364}%
\special{pa 5672 1424}%
\special{pa 5676 1400}%
\special{pa 5700 1394}%
\special{pa 5636 1364}%
\special{fp}%
%
\special{pn 13}%
\special{pa 5554 1898}%
\special{pa 5330 1696}%
\special{fp}%
\special{sh 1}%
\special{pa 5330 1696}%
\special{pa 5366 1756}%
\special{pa 5370 1732}%
\special{pa 5394 1726}%
\special{pa 5330 1696}%
\special{fp}%
%
\special{pn 13}%
\special{pa 5260 2260}%
\special{pa 5038 2060}%
\special{fp}%
\special{sh 1}%
\special{pa 5038 2060}%
\special{pa 5074 2118}%
\special{pa 5078 2096}%
\special{pa 5100 2090}%
\special{pa 5038 2060}%
\special{fp}%
%
\special{pn 13}%
\special{pa 4968 2596}%
\special{pa 4744 2396}%
\special{fp}%
\special{sh 1}%
\special{pa 4744 2396}%
\special{pa 4780 2454}%
\special{pa 4784 2432}%
\special{pa 4808 2426}%
\special{pa 4744 2396}%
\special{fp}%
\put(33.8000,-16.7000){\makebox(0,0)[lb]{$\gamma$ $\rightarrow$ null geodesic}}%
%
\special{pn 4}%
\special{ar 2290 2950 320 80  0.0000000 6.2831853}%
%
\special{pn 4}%
\special{pa 2620 2950}%
\special{pa 2290 3290}%
\special{fp}%
%
\special{pn 4}%
\special{ar 2410 630 300 50  0.0000000 6.2831853}%
%
\special{pn 4}%
\special{pa 2120 630}%
\special{pa 2390 990}%
\special{fp}%
%
\special{pn 13}%
\special{pa 5850 1570}%
\special{pa 6080 1280}%
\special{fp}%
\special{sh 1}%
\special{pa 6080 1280}%
\special{pa 6024 1320}%
\special{pa 6048 1322}%
\special{pa 6054 1346}%
\special{pa 6080 1280}%
\special{fp}%
%
\special{pn 13}%
\special{pa 5560 1900}%
\special{pa 5790 1610}%
\special{fp}%
\special{sh 1}%
\special{pa 5790 1610}%
\special{pa 5734 1650}%
\special{pa 5758 1652}%
\special{pa 5764 1676}%
\special{pa 5790 1610}%
\special{fp}%
%
\special{pn 13}%
\special{pa 5260 2260}%
\special{pa 5490 1970}%
\special{fp}%
\special{sh 1}%
\special{pa 5490 1970}%
\special{pa 5434 2010}%
\special{pa 5458 2012}%
\special{pa 5464 2036}%
\special{pa 5490 1970}%
\special{fp}%
%
\special{pn 13}%
\special{pa 4970 2600}%
\special{pa 5200 2310}%
\special{fp}%
\special{sh 1}%
\special{pa 5200 2310}%
\special{pa 5144 2350}%
\special{pa 5168 2352}%
\special{pa 5174 2376}%
\special{pa 5200 2310}%
\special{fp}%
%
\special{pn 8}%
\special{pa 3320 1760}%
\special{pa 4520 1760}%
\special{fp}%
\special{sh 1}%
\special{pa 4520 1760}%
\special{pa 4454 1740}%
\special{pa 4468 1760}%
\special{pa 4454 1780}%
\special{pa 4520 1760}%
\special{fp}%
\put(50.7000,-28.9000){\makebox(0,0)[rt]{null geodesic}}%
\put(17.5000,-34.1000){\makebox(0,0)[rt]{$\dot{\gamma}(s_2)=A(s_2)$}}%
%
\special{pn 4}%
\special{pa 1990 1480}%
\special{pa 3040 210}%
\special{fp}%
\special{pa 2250 4440}%
\special{pa 3080 3140}%
\special{fp}%
%
\special{pn 8}%
\special{pa 2090 4130}%
\special{pa 2098 4098}%
\special{pa 2104 4066}%
\special{pa 2114 4036}%
\special{pa 2124 4006}%
\special{pa 2136 3978}%
\special{pa 2154 3950}%
\special{pa 2172 3924}%
\special{pa 2194 3902}%
\special{pa 2218 3880}%
\special{pa 2244 3858}%
\special{pa 2270 3840}%
\special{pa 2298 3822}%
\special{pa 2326 3806}%
\special{pa 2354 3790}%
\special{pa 2382 3776}%
\special{pa 2390 3770}%
\special{sp -0.045}%
%
\special{pn 13}%
\special{pa 2400 990}%
\special{pa 2670 670}%
\special{fp}%
\special{sh 1}%
\special{pa 2670 670}%
\special{pa 2612 708}%
\special{pa 2636 712}%
\special{pa 2642 734}%
\special{pa 2670 670}%
\special{fp}%
%
\special{pn 13}%
\special{pa 2400 990}%
\special{pa 2180 600}%
\special{fp}%
\special{sh 1}%
\special{pa 2180 600}%
\special{pa 2196 668}%
\special{pa 2206 646}%
\special{pa 2230 648}%
\special{pa 2180 600}%
\special{fp}%
%
\special{pn 8}%
\special{ar 1780 2990 510 470  0.4275689 0.4520587}%
\special{ar 1780 2990 510 470  0.5255281 0.5500179}%
\special{ar 1780 2990 510 470  0.6234872 0.6479770}%
\special{ar 1780 2990 510 470  0.7214464 0.7459362}%
\special{ar 1780 2990 510 470  0.8194056 0.8438954}%
\special{ar 1780 2990 510 470  0.9173648 0.9418546}%
\special{ar 1780 2990 510 470  1.0153240 1.0398138}%
\special{ar 1780 2990 510 470  1.1132832 1.1377730}%
\special{ar 1780 2990 510 470  1.2112423 1.2357321}%
\special{ar 1780 2990 510 470  1.3092015 1.3336913}%
\special{ar 1780 2990 510 470  1.4071607 1.4316505}%
\special{ar 1780 2990 510 470  1.5051199 1.5296097}%
\end{picture}%
}\hspace{4.5truecm}}

\vspace{1truecm}

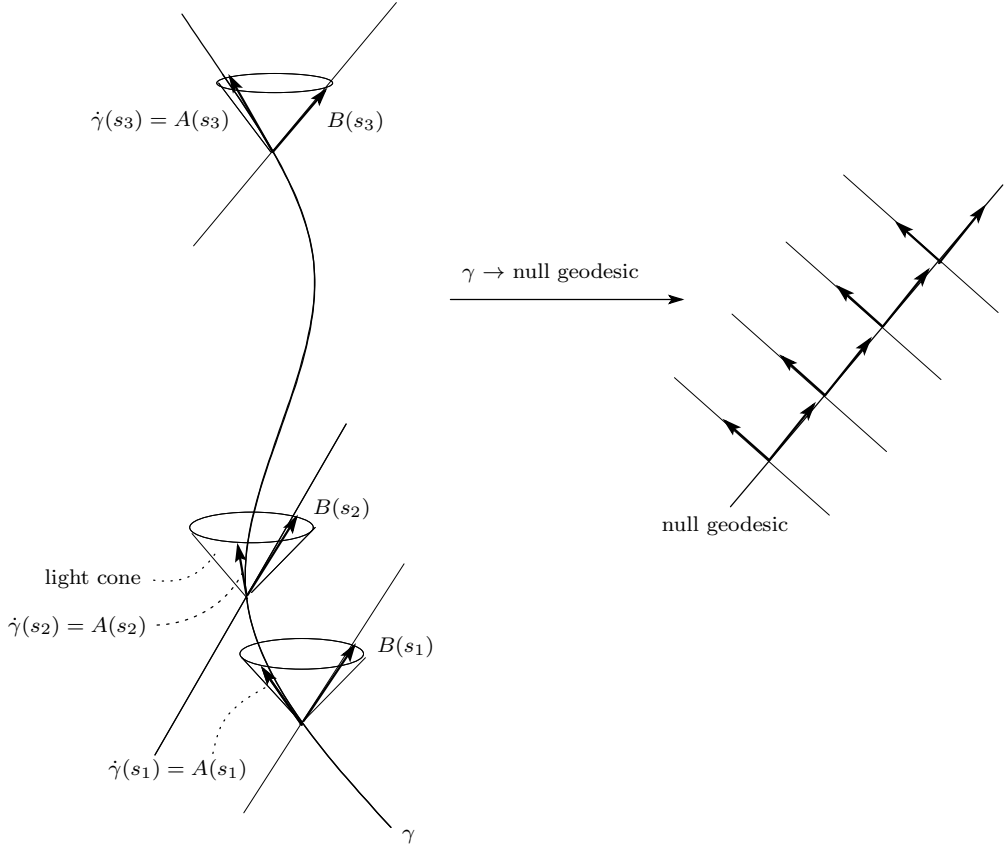
\centerline{Figure 2 : B-scroll over $\gamma$}

\vspace{1truecm}


Therefore, we may identify  $\bigwedge^{n+1}\mathbb{E}^m_s$ with the pseudo-Euclidean space $\mathbb{E}^N_q$ for some positive integer $q$, where $N=\binom{m}{n+1}$. 
The Grassmannian manifold $G(n+1,m)_{t+1}$ can be imbedded into a pseudo-Euclidean space $\bigwedge^{n+1} \mathbb{E}^m_s\simeq \mathbb{E}^N_q$ 
by assigning $\Pi \in G(n+1,m)_{t+1}$ to $\widetilde{e}_1\wedge \cdots \wedge \widetilde{e}_{n+1}$ 
where $(\widetilde{e}_1,\cdots ,\widetilde{e}_{n+1})$ is an orthonormal basis of $\Pi$ compatible with the orientation of $\Pi$. 

Let ${\bf x}:M_t\hookrightarrow  \mathbb{H}^{m-1}_{s-1} \subset \mathbb{E}^m_s$ be an isometric immersion. 
For the immersion {\bf x}, we define a map $\tilde{\nu}:M_t\rightarrow G(n+1,m)_{t+1}$ 
by 
$$\tilde{\nu}(p)={\bf x}(p)\wedge {\bf x}_*(e_1^p)\wedge {\bf x}_*(e_2^p)\wedge \cdots \wedge {\bf x}_*(e_n^p) \quad (p\in M_t),$$
where $(e_1^p,\cdots ,e_n^p)$ is an orthonormal frame of $T_pM_t$ compatible with the orientation of $M_t$. 
This map $\tilde{\nu}$ is called {\it the pseudo-hyperbolic Gauss map of {\bf x}.} 
In the sequel, we all rewrite ${\bf x}_*(e_i)$ by $e_i$. \\ 
Let $(e_1 ,\cdots ,e_n)$ be  an local orthonormal frame field of $TM_t$ compatible with the orientation of $M_t$ 
and  $(e_{n+1},\cdots,e_m)$ be an local orthonormal frame field of $T^\perp M_t$ 
 defined an open set $U$ of $M_t$, respectively.\\
The first derivative of the pseudo-hyperbolic Gauss map $\tilde{\nu}$ is given by
\begin{equation}
e_i\tilde{\nu}=\sum_{k=1}^{n}\sum_{r=n+1}^{m-1}\epsilon_rh^r_{ik}{\bf x}\wedge e_1\wedge \cdots \wedge \underbrace{e_r}_{k-th}\wedge \cdots \wedge e_n. \label{eq4}
\end{equation}

Ye\v{g}in and Dursun proved the following fact.
\begin{lemm} 
\cite{YD} Let $M_t$ be an $n$-dimensional oriented pseudo-Riemannian submanifold of index $t$ of a pseudo-hyperbolic $\mathbb{H}^{m-1}_s\subset \mathbb{E}^m_{s+1}$. 
Then the Laplacian of the pseudo-hyperbolic Gauss map $\tilde{\nu}:M_t\rightarrow G(n+1,m)\subset \mathbb{E}^N_q, N=\binom{m}{n+1}$. for some $q$ is given by
\begin{equation}
\begin{split}
\Delta \tilde{\nu}=\Vert \hat{h}\Vert ^2\tilde{\nu}
+n\hat{H}\wedge e_1\wedge \cdots \wedge e_n-n\sum_{k=1}^n{\bf x}\wedge e_1 \wedge \cdots \wedge \underbrace{D_{e_k}\hat{H}}_{k-th}\wedge \cdots \wedge e_n\\
+\sum_{\substack{i,k=1 \\ j\neq k}}^n\sum_{\substack{r,s=n+1 \\ r<s}}^{m-1}\epsilon_r\epsilon_s R^r_{sjk}{\bf x}\wedge e_1\wedge \cdots \wedge \underbrace{e_r}_{j-th} \wedge \cdots \wedge \underbrace{e_s}_{k-th} \wedge \cdots \wedge e_n
\end{split} \label{eq1}
\end{equation}
where $R^r_{sjk}=R^D(e_j,e_k;e_r,e_s)$.
\end{lemm}

In case of $n=m-2$, Ye\v{g}in and Dursun have the following fact. 
\begin{lemm}
\cite{YD} For an oriented pseudo-Riemannian hypersurface $M_t$ with index $t$ of $\mathbb{H}^{n+1}_{s-1}\subset \mathbb{E}^{n+2}_s$ we have
\begin{equation}
\Delta (e_{n+1}\wedge e_1\wedge e_2 \wedge \cdots \wedge e_n)=-n \hat{\mathcal{H}} \tilde{\nu}-ne_{n+1}\wedge e_1\wedge e_2\wedge \cdots \wedge e_n \label{eq3}
\end{equation}
where $\hat{\mathcal{H}}$ is the mean curvature $M_t$ in $\mathbb{H}^{n+1}_{s-1}$, that is, $\hat{H}=\epsilon_{n+1} \hat{\mathcal{H}} e_{n+1}$.
\end{lemm}

\begin{lemm}
{\rm If there exists  a polynomial $P(t)=(t-\lambda_1)(t-\lambda_2)$ with mutually distinct roots $\lambda_1, \lambda_2 \in \mathbb{C}$ 
such that 
$P(\Delta)\tilde{\nu}=0$, 
then $\tilde{\nu}$ is of at most $2$-type or infinite type. 
}\label{lemm2}
\end{lemm}
\begin{proof}
Assume that $\tilde{\nu}$ is of finite type and it has the following spectral decomposition; 
$\tilde{\nu}=\tilde{\nu}_{\hat{\lambda}_1}+\cdots +\tilde{\nu}_{\hat{\lambda}_k}$ 
with $\Delta \tilde{\nu}_{\hat{\lambda}_i}=\hat{\lambda}_i\tilde{\nu}_{\hat{\lambda}_i}$ $(1\leq i\leq k)$. 
Then we have 
\begin{eqnarray}
0=P(\Delta)\tilde{\nu}&=&P(\Delta)(\tilde{\nu}_{\hat{\lambda}_1}+\cdots +\tilde{\nu}_{\hat{\lambda}_k})\nonumber \\
&=& \sum_{i=1}^{k}(\hat{\lambda}_i-\hat{\lambda}_1)(\hat{\lambda}_i-\hat{\lambda}_2)\tilde{\nu}_{\hat{\lambda}_i}. \nonumber 
\end{eqnarray}
Thus, we have $(\hat{\lambda}_i-\lambda_1)(\hat{\lambda}_i-\lambda_2)=0$ $(1\leq i\leq k)$. 
Hence, we have $\hat{\lambda}_i=\lambda_1$ or $\lambda_2$ for all $i$, that is, $\tilde{\nu}$ is of at most $2$-type. 
Hence the statement of this lemma follows.
\end{proof}

\section{Proof of (\rnum{1}) and (\rnum{2}) of Theorem \ref{thmA}.}

In this section, we prove (\rnum{1}) and (\rnum{2}) of Theorem \ref{thmA}. 

\vspace{0.5truecm}

\noindent
{\it Proof of (\rnum{1}) and (\rnum{2}) of Theorem \ref{thmA}.} 
Let $M$ be the complex circle of radius $\kappa$ in $\mathbb{H}^3_1$. 
This surface $M$ is parameterized as
$${\bf x}(z)=\sqrt{\kappa} (\cos z,\sin z) \quad (z\in \mathbb{C}),$$
where $\sqrt{\kappa}$ is one (with smaller argument) of squared roots of $\kappa$.  
Note that $M$ is included by $\mathbb{H}^3_1$ because of $\mathrm{Re}(\kappa)=-1$.  
For the convenience, we put $\sqrt{\kappa}=c+\sqrt{-1}d$.  Let $z=x+\sqrt{-1}y$.  
By simple calculations, we have 
$$\langle {\bf x}_x , {\bf x}_x \rangle =-1 ,\quad \langle {\bf x}_x , {\bf x}_y \rangle =-2cd ,\quad 
\langle {\bf x}_y , {\bf x}_y \rangle =1.$$
Also, we can show that the unit normal vector field $N$ of $M$ in $\mathbb{H}^3_1$ 
is given by $N=(d+\sqrt{-1}c)(\cos z,\sin z)$.  
With respect to the frame field $({\bf x}_x,{\bf x}_y)$, the shape operator $A_N$ in the direction $N$ is expressed 
as 
$$A_N=\left(
\begin{array}{cc}
\alpha & -\beta \\
\beta & \alpha
\end{array}
\right),
$$
where $\alpha=-\frac{2cd}{c^2+d^2}$ and $\beta=\frac{1}{c^2+d^2}$. 
Put $e_1:={\bf x}_x$, $\tilde{e}_2={\bf x}_y-\langle {\bf x}_y,e_1 \rangle e_1$ and $e_2=\frac{1}{|\tilde{e}_2|}\tilde{e}_2$. 
Then $(e_1,e_2)$ forms an orthonormal tangent frame field  on $M$.  
Note that $\langle e_1, e_1\rangle =-1$, $\langle e_2, e_2\rangle =1$ and $\vert \tilde{e}_2\vert^2=1+4c^2d^2$. 
With respect to $(e_1,e_2)$, the shape operator $A_N$ is expressed as 
$$A_N=
\left(
\begin{array}{cc}
0 & |\tilde{e}_2|\beta \\
-|\tilde{e}_2|\beta & 2\alpha
\end{array}
\right).
$$
Thus, by remarking $\alpha=-2cd\beta$ and $c^2-d^2=-1$, we obtain $\hat{\mathcal{H}}=\alpha$ and 
$\Vert \hat{h}\Vert^2=2(\alpha^2-\beta^2)$, 
where $\hat{\mathcal{H}}$ and $\hat{h}$ denote the mean curvature and the second fundamental form of $M$ in 
$\mathbb{H}^3_1$. 
Hence, by (\ref{eq5}), $M$ is flat. 
By (\ref{eq1}) and (\ref{eq3}), we have 
\begin{eqnarray}
\Delta \tilde{\nu}&=&2(\alpha^2-\beta^2)\tilde{\nu}+2\alpha N\wedge e_1\wedge e_2, \label{eq11} \\
\Delta^2 \tilde{\nu}&=&4((\alpha^2-\beta^2)^2-\alpha^2)\tilde{\nu}+4\alpha(\alpha^2-\beta^2-1)N\wedge e_1\wedge e_2. \label{eq12} 
\end{eqnarray}
Hence,
\begin{eqnarray}
\Delta^2\tilde{\nu}-2(\alpha^2-\beta^2-1)\Delta \tilde{\nu}+4\beta^2\tilde{\nu}=0.
\end{eqnarray}
Therefore by Lemma \ref{lemm2}, $\tilde{\nu}$ is either of finite type with type number $k\leq 2$ or of infinite type. \\
If $\kappa=-1$, then we have $c=0$ and hence $\alpha=0$. 
Therefore, it follows from (\ref{eq11}) that $\Delta \tilde{\nu}=-2\beta^2 \tilde{\nu}$, that is, $\tilde{\nu}$ is of $1$-type. 
If $\kappa \neq -1$, then we have $\alpha \neq 0$. 
Hence it follows from (\ref{eq11}) that $\tilde{\nu}$ is of not $1$-type. 
Therefore, $\tilde{\nu}$ is of $2$-type or of infinite type. 
Suppose that it is of $2$-type, and that 
$\tilde{\nu}$ has decomposition 
$\tilde{\nu}=\tilde{\nu}_1+\tilde{\nu}_2$ 
($\Delta{\tilde{\nu}_1}=\lambda_1\tilde{\nu}_1$, $\Delta{\tilde{\nu}_2}=\lambda_2\tilde{\nu}_2$).  
where $\lambda_1 , \lambda_2 \in \mathbb{R}$ are mutually distinct.  
From (\ref{eq11}) and (\ref{eq12}), $\tilde{\nu}_1$ and $\tilde{\nu}_2$ can be expressed as 
\begin{eqnarray}
\tilde{\nu}_1&=&a\tilde{\nu}+bN\wedge e_1\wedge e_2, \label{eq13}\\
\tilde{\nu}_2&=&(1-a)\tilde{\nu}-bN\wedge e_1\wedge e_2 \label{eq14}
\end{eqnarray}
for some constants $a$ and $b$.  
By substituting (\ref{eq13}) and (\ref{eq14}) into (\ref{eq11}) and (\ref{eq12}), 
and comparing coefficient of $\tilde{\nu}$ and $N\wedge e_1\wedge e_2$, we have 
\begin{eqnarray}
\left\{
\begin{array}{l}
(\lambda_1-\lambda_2)a=2(\alpha^2-\beta^2)-\lambda_2 \\
(\lambda_1-\lambda_2)b=2\alpha \\
(\lambda_1^2-\lambda_2^2)a=4(\alpha^2-\beta^2)^2-4\alpha^2-\lambda_2^2 \\
(\lambda_1^2-\lambda_2^2)b=4\alpha(\alpha^2-\beta^2-1)\\
\end{array}
\right.
\label{eq16}
\end{eqnarray}
Therefore, we obtain 
\begin{equation}
\lambda_2^2-2(\alpha^2-\beta^2-1)\lambda_2+4\beta^2=0. \label{eq15}
\end{equation}
The discriminant of (\ref{eq15}) is
\begin{equation}
\begin{split}
(\alpha^2-\beta^2-1)^2-4\beta^2=&(\alpha^2-\beta^2)^2-1\\
=&-\frac{1}{(c^2+d^2)^4}{(16c^8+16c^6+32c^4+16c^2)}\\
<&0.
\end{split}
\end{equation}
Therefore, there is no $\lambda_2 \in \mathbb{R}$ satisfying (\ref{eq15}). 
Thus a contradiction arises. 
Therefore, $\tilde{\nu}$ is of infinite type.
\qed

\section{Proof of (\rnum{3}) and (\rnum{4}) of Theorem \ref{thmA}.}

In this section, we prove (\rnum{3}) and (\rnum{4}) of Theorem \ref{thmA}.

\vspace{0.5truecm}

\noindent
{\it Proof of (\rnum{3}) and (\rnum{4}) of Theorem \ref{thmA}.} 
Let $(A,B,C)$ be the Cartan frame field along a null curve $\gamma(s)$ in $\mathbb{H}^3_1$ given by Definition \ref{defi1}. 
Then, the immersion ${\bf x}(s,t)=\gamma(s)+tB(s)$ parametrizes the B-scroll over a null curve $\gamma$. 
We have 
$${\bf x}_s(s,t)=A(s)+t(k_2C(s)-\gamma(s)) \quad \mathrm{and} \quad {\bf x}_t(s,t)=B(s)$$
and hence
$$\langle {\bf x}_s,{\bf x}_s\rangle =t^2(k_2^2-1), \quad \langle {\bf x}_s,{\bf x}_t\rangle=-1 \quad \mathrm{and} \quad \langle {\bf x}_t,{\bf x}_t\rangle=0.$$
The unit normal vector field of ${\bf x}$ is given by $N(s,t)=k_2tB(s)+C(s)$. 
With respect to the frame $({\bf x_s},{\bf x_t})$, the shape operator $A_N$ in the direction $N$ is expressed as
$$A_N=
\left(
\begin{array}{cc}
-k_2 & 0 \\
-k_1(s) & -k_2
\end{array}
\right)
.$$
%
When $|k_2| > 1$ and $t\neq0$, an orthonormal frame field $(e_1, e_2)$ on $M$ is given by
$$e_1=\frac{{\bf x_s}}{|{\bf x_s}|} ,\quad e_2=e_1+|{\bf x_s}|{\bf x_t}.$$
By simple calculations, we have 
\begin{eqnarray}
A_N(e_1)&=&\frac{1}{|{\bf x_s}|}A_N({\bf x_s}) \nonumber \\
&=&\frac{1}{|{\bf x_s}|}(-k_2{\bf x_s}-k_1(s){\bf x_t}) \nonumber \\
&=&-\frac{1}{|{\bf x_s}|}\left( k_2|{\bf x_s}|e_1+k_1(s)\frac{1}{|{\bf x_s}|}(e_2-e_1)\right) \nonumber \\
&=& \left(-k_2+\frac{k_1(s)}{t^2|k_2^2-1|}\right)e_1- \frac{k_1(s)}{t^2|k_2^2-1|}e_2, \nonumber \\
A_N(e_2)&=&A_N(e_1)+\frac{\langle{\bf x_s},{\bf x_s}\rangle}{|{\bf x_s}|}A_N({\bf x_t})\nonumber \\
&=& \frac{k_1(s)}{t^2|k_2^2-1|}e_1+\left(-k_2- \frac{k_1(s)}{t^2|k_2^2-1|}\right)e_2. \nonumber
\end{eqnarray}
Thus, with respect to an orthonormal frame $(e_1,e_2)$, the shape operator $A_N$ is expressed as 
$$A_N=
\left(
\begin{array}{cc}
-k_2+\frac{k_1(s)}{t^2|k_2^2-1|} & \frac{k_1(s)}{t^2|k_2^2-1|} \\
-\frac{k_1(s)}{t^2|k_2^2-1|} & -k_2-\frac{k_1(s)}{t^2|k_2^2-1|}
\end{array}
\right)
$$
and hence we have $\hat{\mathcal{H}}=-k_2$, $\Vert \hat{h}\Vert ^2=2k_2^2$. 
When $|k_2|<1$ and $t \neq 0$, we put $e_1=\frac{{\bf x_s}}{|{\bf x_s}|}$ and $e_2=e_1-|{\bf x_s}|{\bf x_t}$. 
Similarly we have  $\hat{\mathcal{H}}=-k_2$, $\Vert \hat{h}\Vert ^2=2k_2^2$. 
Hence $M$ is non-flat B-scroll by (\ref{eq5}) when $k_2^2\neq 1$ and $t\neq 0$.
We put $e_3:=N$ and 
\begin{eqnarray}
\tilde{\nu}_1:=\frac{1}{k_2^2-1}(-\tilde{\nu}+k_2e_3\wedge e_1\wedge e_2),\\
\tilde{\nu}_2:=\frac{1}{k_2^2-1}(k_2^2\tilde{\nu}-k_2e_3\wedge e_1\wedge e_2).
\end{eqnarray}
It is clear that $\tilde{\nu}=\tilde{\nu}_1+\tilde{\nu}_2$. 
Using (\ref{eq1}) and (\ref{eq3}) we obtain that $\Delta \tilde{\nu}_1=0$ and $\Delta \tilde{\nu}_2=2(k_2^2-1)\tilde{\nu}_2.$
On the other hand, by using (\ref{eq4}) and (\ref{eq3}), we have 
$$e_1(\tilde{\nu})=\epsilon_1\left(-k_2+\frac{k_1(s)}{t^2(k_2^2-1)}\right){\bf x}\wedge e_3\wedge e_2-\epsilon_1\frac{k_1(s)}{t^2(k_2^2-1)}{\bf x}\wedge e_1\wedge e_3$$
$$e_1(e_3\wedge e_1\wedge e_2)=\epsilon_1e_3\wedge {\bf x}\wedge e_2$$
and hence 
$$e_1(\tilde{\nu}_1)=\frac{\epsilon_1 k_1(s)}{t^2(k_2^2-1)^2}({\bf x}\wedge e_3\wedge e_2-{\bf x}\wedge e_1\wedge e_3) \neq 0.$$
Therefore $\tilde{\nu}$ is of null $2$-type.

When $k_2^2=1$ or $t=0$, an orthonormal frame field $(e_1,e_2)$ on $M$ is given by 
$$e_1=\frac{1}{\sqrt{2}}({\bf x_s}+{\bf x_t}) ,\quad e_2=\frac{1}{\sqrt{2}}({\bf x_s}-{\bf x_t}).$$
By simple calculations, we have 
\begin{eqnarray}
A_N(e_1)&=& \left(-k_2-\frac{k_1(s)}{2}\right)e_1+\frac{k_1(s)}{2}e_2,\nonumber \\
A_N(e_2)&=&-\frac{k_1(s)}{2}e_1+\left(-k_2+\frac{k_1(s)}{2}\right)e_2.\nonumber
\end{eqnarray}
Thus, with respect to an orthonormal frame $(e_1,e_2)$ ,the shape operator $A_N$ is expressed as 
$$A_N=
\left(
\begin{array}{cc}
-k_2-\frac{k_1(s)}{2} & -\frac{k_1(s)}{2} \\
\frac{k_1(s)}{2} & -k_2+\frac{k_1(s)}{2}
\end{array}
\right)
$$
and hence we have $\hat{\mathcal{H}}=-k_2$ and $\Vert \hat{h}\Vert ^2=2k_2^2$. 
Hence $M$ is flat B-scroll by (\ref{eq5}).
Using (\ref{eq1}) and (\ref{eq3}), we obtain $\Delta \tilde{\nu}=2\tilde{\nu}-2k_2e_3\wedge e_1\wedge e_2\neq 0$ and $\Delta^2\tilde{\nu}=0$.
Therefore $\tilde{\nu}$ is infinite type by Lemma \ref{lemm1}.

Conversely, assume that the pseudo-hyperbolic Gauss map $\tilde{\nu}$ is of null $2$-type. Then, from (\ref{eq4}), (\ref{eq1}) and (\ref{eq3}), we obtain 
\begin{equation}
\Delta \tilde{\nu}=\Vert \hat{h}\Vert ^2\tilde{\nu}+2\hat{\mathcal{H}} e_3\wedge e_1\wedge e_2, \label{eq9} 
\end{equation}
\begin{equation}
\begin{split}
\Delta^2\tilde{\nu}=&(\Vert \hat{h}\Vert^2 -2)\Delta \tilde{\nu}+(\Delta(\Vert \hat{h}\Vert^2)-4\hat{\mathcal{H}}^2+2\Vert \hat{h}\Vert^2)\tilde{\nu}\\
&-2\sum_{j=1}^{2}\epsilon_je_j(\Vert \hat{h}\Vert^2)h^{3}_{j1}{\bf x}\wedge e_{3}\wedge e_2
-2\sum_{j=1}^{2}\epsilon_je_j(\Vert \hat{h}\Vert^2)h^{3}_{j2}{\bf x}\wedge e_{1}\wedge e_3. \label{eq6}
\end{split}
\end{equation}
Since $\tilde{\nu}$ is of null $2$-type, 
we can put $\tilde{\nu}=\tilde{\nu}_1+\tilde{\nu}_2$ with $\Delta \tilde{\nu}_1=0$ and $\Delta \tilde{\nu}_2=\lambda_2\tilde{\nu}_2$ $(\lambda_2\neq0)$, 
where $\tilde{\nu}_1$ is non-constant. 
Then we have $\Delta^2 \tilde{\nu}=\lambda_2\Delta \tilde{\nu}$. 
This together with (\ref{eq6}) implies, $e_j(\Vert \hat{h}\Vert^2)=0$ (i.e. $\Vert \hat{h}\Vert^2$ is constant), 
$\lambda_2=\Vert \hat{h}\Vert^2-2$.  Hence the Gaussian curvature is constant, 
that is, $\Delta (\Vert \hat{h}\Vert^2)=0$. 
This together with (\ref{eq6}) implies that $\Vert \hat{h}\Vert^2=2\hat{\mathcal{H}}$. 
Hence, from (\ref{eq5}), we have 
$$S=-2+4\hat{\mathcal{H}}^2-\Vert h\Vert^2=-2+\Vert \hat{h}\Vert^2=\lambda_2\neq 0.$$
From Fact \ref{KimKim}, (\rnum{1}) and (\rnum{2}) of Theorem \ref{thmA}, $M$ is an open part of a non-flat B-scroll. 
\qed

\section{Proof of Theorems \ref{thmD} and \ref{thmB}.}

In this section, we consider the pseudo-hyperbolic Gauss map of the parallel surface of a complex circle and a B-scroll. 

\begin{defi}
{\sl Let $\bar{M}$ be a pseudo-Riemannian manifold and $M$ be a pseudo-Riemannian hypersurface of $\bar{M}$  
with unit normal vector field $N$.
At least locally and for $u\in \mathbb{R}$ sufficiently close to $0$, for the map ${\bf x}^u:M\rightarrow \bar{M}$ 
defined by ${\bf x}^u(p):=\exp_{{\bf x}(p)}uN_p=\gamma_{N_p}(u)$ $(p\in M)$ is an immersion and 
$M_u:={\bf x}^u(M)$ is called {\it the parallel surface of $M$ at distance $u$}, 
where $\exp_{{\bf x}(p)}$ denotes the exponential map of $\bar{M}$ at ${\bf x}(p)$ and $\gamma_{N_p}$ denotes 
the geodesic in $\bar{M}$ with $\gamma'_{N_p}(0)=N_p$.}
\end{defi}

\begin{example} 
{\rm Let ${\bf x}:M\hookrightarrow \mathbb{S}^{n+1}_1 (\subset \mathbb{E}^{n+2}_1)$ be a Lorentzian hypersurface and $N$ be its unit normal vector field. 
Then ${\bf x}^u$ is given by ${\bf x}^u(p)=\cos u{\bf x}(p)+\sin uN_p$ $(p\in M)$. }\label{ex1}
\end{example}
\begin{example} 
{\rm Let ${\bf x}:M\hookrightarrow \mathbb{H}^{n+1}_1 (\subset \mathbb{E}^{n+2}_2)$ be a Lorentzian hypersurface and $N$ be its unit normal vector field. 
Then ${\bf x}^u$ is given by ${\bf x}^u(p)=\cosh u{\bf x}(p)+\sinh uN_p$ $(p\in M)$. }
\end{example}

\vspace{1truecm}

\centerline{
\small{
\unitlength 0.1in
\begin{picture}( 35.8000, 33.2000)(  8.4000,-34.9000)
\put(27.3000,-19.5000){\makebox(0,0)[rt]{O}}%
%
\special{pn 8}%
\special{pa 1020 1940}%
\special{pa 4420 1940}%
\special{fp}%
\special{sh 1}%
\special{pa 4420 1940}%
\special{pa 4354 1920}%
\special{pa 4368 1940}%
\special{pa 4354 1960}%
\special{pa 4420 1940}%
\special{fp}%
\special{pn 8}%
\special{pa 1128 220}%
\special{pa 1262 344}%
\special{pa 1384 458}%
\special{pa 1498 562}%
\special{pa 1606 656}%
\special{pa 1704 744}%
\special{pa 1796 822}%
\special{pa 1884 894}%
\special{pa 1964 960}%
\special{pa 2040 1018}%
\special{pa 2112 1072}%
\special{pa 2180 1120}%
\special{pa 2244 1162}%
\special{pa 2304 1198}%
\special{pa 2362 1232}%
\special{pa 2418 1260}%
\special{pa 2472 1282}%
\special{pa 2524 1302}%
\special{pa 2574 1318}%
\special{pa 2624 1330}%
\special{pa 2672 1336}%
\special{pa 2720 1340}%
\special{pa 2768 1340}%
\special{pa 2816 1336}%
\special{pa 2864 1328}%
\special{pa 2914 1316}%
\special{pa 2964 1300}%
\special{pa 3016 1280}%
\special{pa 3070 1256}%
\special{pa 3126 1228}%
\special{pa 3184 1194}%
\special{pa 3246 1156}%
\special{pa 3310 1112}%
\special{pa 3378 1064}%
\special{pa 3450 1010}%
\special{pa 3528 950}%
\special{pa 3610 884}%
\special{pa 3696 812}%
\special{pa 3790 732}%
\special{pa 3890 644}%
\special{pa 3998 548}%
\special{pa 4114 442}%
\special{pa 4238 328}%
\special{pa 4240 326}%
\special{pa 4240 324}%
\special{pa 4242 324}%
\special{pa 4244 322}%
\special{pa 4244 322}%
\special{pa 4246 320}%
\special{pa 4246 318}%
\special{pa 4248 318}%
\special{pa 4250 316}%
\special{pa 4250 316}%
\special{pa 4252 314}%
\special{pa 4254 312}%
\special{pa 4254 312}%
\special{pa 4256 310}%
\special{pa 4258 310}%
\special{pa 4258 308}%
\special{pa 4260 306}%
\special{pa 4262 306}%
\special{pa 4262 304}%
\special{pa 4264 304}%
\special{pa 4266 302}%
\special{pa 4266 300}%
\special{pa 4268 300}%
\special{pa 4270 298}%
\special{pa 4270 298}%
\special{pa 4272 296}%
\special{pa 4274 294}%
\special{pa 4274 294}%
\special{pa 4276 292}%
\special{pa 4278 290}%
\special{pa 4278 290}%
\special{pa 4280 288}%
\special{pa 4280 288}%
\special{pa 4282 286}%
\special{pa 4284 284}%
\special{pa 4284 284}%
\special{pa 4286 282}%
\special{pa 4288 282}%
\special{pa 4288 280}%
\special{pa 4290 278}%
\special{pa 4292 278}%
\special{pa 4292 276}%
\special{pa 4294 274}%
\special{pa 4296 274}%
\special{pa 4296 272}%
\special{pa 4298 272}%
\special{pa 4300 270}%
\special{pa 4300 268}%
\special{pa 4302 268}%
\special{pa 4304 266}%
\special{pa 4304 264}%
\special{pa 4306 264}%
\special{pa 4308 262}%
\special{pa 4308 262}%
\special{pa 4310 260}%
\special{pa 4312 258}%
\special{pa 4314 258}%
\special{pa 4314 256}%
\special{pa 4316 254}%
\special{pa 4318 254}%
\special{pa 4318 252}%
\special{pa 4320 252}%
\special{pa 4322 250}%
\special{pa 4322 248}%
\special{pa 4324 248}%
\special{pa 4326 246}%
\special{pa 4326 244}%
\special{pa 4328 244}%
\special{pa 4330 242}%
\special{pa 4330 240}%
\special{pa 4332 240}%
\special{pa 4334 238}%
\special{pa 4334 238}%
\special{pa 4336 236}%
\special{pa 4338 234}%
\special{pa 4338 234}%
\special{pa 4340 232}%
\special{pa 4342 230}%
\special{pa 4342 230}%
\special{pa 4344 228}%
\special{pa 4346 226}%
\special{pa 4346 226}%
\special{pa 4348 224}%
\special{pa 4350 224}%
\special{pa 4352 222}%
\special{pa 4352 220}%
\special{sp}%
\special{pn 8}%
\special{pa 1312 3490}%
\special{pa 1430 3380}%
\special{pa 1542 3280}%
\special{pa 1646 3188}%
\special{pa 1742 3106}%
\special{pa 1832 3030}%
\special{pa 1916 2960}%
\special{pa 1996 2898}%
\special{pa 2070 2840}%
\special{pa 2140 2790}%
\special{pa 2206 2744}%
\special{pa 2268 2704}%
\special{pa 2328 2668}%
\special{pa 2384 2638}%
\special{pa 2440 2612}%
\special{pa 2492 2590}%
\special{pa 2544 2572}%
\special{pa 2594 2558}%
\special{pa 2642 2548}%
\special{pa 2692 2542}%
\special{pa 2740 2540}%
\special{pa 2788 2542}%
\special{pa 2836 2548}%
\special{pa 2884 2558}%
\special{pa 2934 2572}%
\special{pa 2986 2588}%
\special{pa 3038 2610}%
\special{pa 3092 2636}%
\special{pa 3150 2666}%
\special{pa 3208 2702}%
\special{pa 3272 2742}%
\special{pa 3338 2786}%
\special{pa 3406 2838}%
\special{pa 3480 2894}%
\special{pa 3560 2956}%
\special{pa 3644 3024}%
\special{pa 3734 3100}%
\special{pa 3830 3184}%
\special{pa 3932 3274}%
\special{pa 4044 3374}%
\special{pa 4162 3484}%
\special{pa 4164 3484}%
\special{pa 4164 3486}%
\special{pa 4166 3486}%
\special{pa 4168 3488}%
\special{pa 4168 3490}%
\special{pa 4170 3490}%
\special{sp}%
\put(8.4000,-6.1000){\makebox(0,0)[lb]{${\LARGE \mathbb{H}^3_1}$}}%
%
\special{pn 8}%
\special{ar 2740 1940 160 600  1.5707963 1.6023753}%
\special{ar 2740 1940 160 600  1.6971121 1.7286911}%
\special{ar 2740 1940 160 600  1.8234279 1.8550069}%
\special{ar 2740 1940 160 600  1.9497437 1.9813226}%
\special{ar 2740 1940 160 600  2.0760595 2.1076384}%
\special{ar 2740 1940 160 600  2.2023753 2.2339542}%
\special{ar 2740 1940 160 600  2.3286911 2.3602700}%
\special{ar 2740 1940 160 600  2.4550069 2.4865858}%
\special{ar 2740 1940 160 600  2.5813226 2.6129016}%
\special{ar 2740 1940 160 600  2.7076384 2.7392174}%
\special{ar 2740 1940 160 600  2.8339542 2.8655332}%
\special{ar 2740 1940 160 600  2.9602700 2.9918490}%
\special{ar 2740 1940 160 600  3.0865858 3.1181647}%
\special{ar 2740 1940 160 600  3.2129016 3.2444805}%
\special{ar 2740 1940 160 600  3.3392174 3.3707963}%
\special{ar 2740 1940 160 600  3.4655332 3.4971121}%
\special{ar 2740 1940 160 600  3.5918490 3.6234279}%
\special{ar 2740 1940 160 600  3.7181647 3.7497437}%
\special{ar 2740 1940 160 600  3.8444805 3.8760595}%
\special{ar 2740 1940 160 600  3.9707963 4.0023753}%
\special{ar 2740 1940 160 600  4.0971121 4.1286911}%
\special{ar 2740 1940 160 600  4.2234279 4.2550069}%
\special{ar 2740 1940 160 600  4.3497437 4.3813226}%
\special{ar 2740 1940 160 600  4.4760595 4.5076384}%
\special{ar 2740 1940 160 600  4.6023753 4.6339542}%
%
\special{pn 13}%
\special{pa 2740 1940}%
\special{pa 2740 1350}%
\special{fp}%
\special{sh 1}%
\special{pa 2740 1350}%
\special{pa 2720 1418}%
\special{pa 2740 1404}%
\special{pa 2760 1418}%
\special{pa 2740 1350}%
\special{fp}%
%
\special{pn 8}%
\special{ar 2740 1940 160 600  4.7467397 6.2831853}%
\special{ar 2740 1940 160 600  0.0000000 1.5707963}%
%
\special{pn 13}%
\special{pa 2740 1340}%
\special{pa 3320 1330}%
\special{fp}%
\special{sh 1}%
\special{pa 3320 1330}%
\special{pa 3254 1312}%
\special{pa 3268 1332}%
\special{pa 3254 1352}%
\special{pa 3320 1330}%
\special{fp}%
\put(25.6000,-12.4000){\makebox(0,0)[lb]{${\bf x}(p)$}}%
\put(33.5000,-14.2000){\makebox(0,0)[lb]{$N(p)$}}%
%
\special{pn 8}%
\special{ar 3640 1940 170 1060  1.5707963 1.5903085}%
\special{ar 3640 1940 170 1060  1.6488451 1.6683573}%
\special{ar 3640 1940 170 1060  1.7268939 1.7464061}%
\special{ar 3640 1940 170 1060  1.8049427 1.8244549}%
\special{ar 3640 1940 170 1060  1.8829914 1.9025036}%
\special{ar 3640 1940 170 1060  1.9610402 1.9805524}%
\special{ar 3640 1940 170 1060  2.0390890 2.0586012}%
\special{ar 3640 1940 170 1060  2.1171378 2.1366500}%
\special{ar 3640 1940 170 1060  2.1951866 2.2146988}%
\special{ar 3640 1940 170 1060  2.2732354 2.2927475}%
\special{ar 3640 1940 170 1060  2.3512841 2.3707963}%
\special{ar 3640 1940 170 1060  2.4293329 2.4488451}%
\special{ar 3640 1940 170 1060  2.5073817 2.5268939}%
\special{ar 3640 1940 170 1060  2.5854305 2.6049427}%
\special{ar 3640 1940 170 1060  2.6634793 2.6829914}%
\special{ar 3640 1940 170 1060  2.7415280 2.7610402}%
\special{ar 3640 1940 170 1060  2.8195768 2.8390890}%
\special{ar 3640 1940 170 1060  2.8976256 2.9171378}%
\special{ar 3640 1940 170 1060  2.9756744 2.9951866}%
\special{ar 3640 1940 170 1060  3.0537232 3.0732354}%
\special{ar 3640 1940 170 1060  3.1317719 3.1512841}%
\special{ar 3640 1940 170 1060  3.2098207 3.2293329}%
\special{ar 3640 1940 170 1060  3.2878695 3.3073817}%
\special{ar 3640 1940 170 1060  3.3659183 3.3854305}%
\special{ar 3640 1940 170 1060  3.4439671 3.4634793}%
\special{ar 3640 1940 170 1060  3.5220158 3.5415280}%
\special{ar 3640 1940 170 1060  3.6000646 3.6195768}%
\special{ar 3640 1940 170 1060  3.6781134 3.6976256}%
\special{ar 3640 1940 170 1060  3.7561622 3.7756744}%
\special{ar 3640 1940 170 1060  3.8342110 3.8537232}%
\special{ar 3640 1940 170 1060  3.9122597 3.9317719}%
\special{ar 3640 1940 170 1060  3.9903085 4.0098207}%
\special{ar 3640 1940 170 1060  4.0683573 4.0878695}%
\special{ar 3640 1940 170 1060  4.1464061 4.1659183}%
\special{ar 3640 1940 170 1060  4.2244549 4.2439671}%
\special{ar 3640 1940 170 1060  4.3025036 4.3220158}%
\special{ar 3640 1940 170 1060  4.3805524 4.4000646}%
\special{ar 3640 1940 170 1060  4.4586012 4.4781134}%
\special{ar 3640 1940 170 1060  4.5366500 4.5561622}%
\special{ar 3640 1940 170 1060  4.6146988 4.6342110}%
\special{ar 3640 1940 170 1060  4.6927475 4.7122597}%
%
\special{pn 8}%
\special{ar 3640 1930 170 1060  4.7123890 6.2831853}%
\special{ar 3640 1930 170 1060  0.0000000 1.5707963}%
\put(25.4000,-28.2000){\makebox(0,0)[lb]{${\bf x}(p')$}}%
\put(34.9000,-7.3000){\makebox(0,0)[lb]{$\gamma_{N_p}(u)$}}%
\put(31.9000,-31.8000){\makebox(0,0)[lb]{$\gamma_{N_{p'}}(u)$}}%
\put(21.6000,-22.8000){\makebox(0,0)[lb]{${\bf x}(M)$}}%
\put(41.6000,-23.5000){\makebox(0,0)[lb]{${\bf x}^u(M)$}}%
\put(42.4000,-36.0000){\makebox(0,0)[lb]{$\gamma_{N_{p'}}$}}%
\put(44.1000,-3.4000){\makebox(0,0)[lb]{$\gamma_{N_{p}}$}}%
%
\special{pn 13}%
\special{pa 2760 1940}%
\special{pa 3340 1930}%
\special{fp}%
\special{sh 1}%
\special{pa 3340 1930}%
\special{pa 3274 1912}%
\special{pa 3288 1932}%
\special{pa 3274 1952}%
\special{pa 3340 1930}%
\special{fp}%
%
\special{pn 8}%
\special{pa 3340 1330}%
\special{pa 3340 1940}%
\special{dt 0.045}%
%
\special{pn 20}%
\special{sh 1}%
\special{ar 2730 1340 10 10 0  6.28318530717959E+0000}%
\special{sh 1}%
\special{ar 2730 1330 10 10 0  6.28318530717959E+0000}%
\special{sh 1}%
\special{ar 2730 1330 10 10 0  6.28318530717959E+0000}%
%
\special{pn 20}%
\special{sh 1}%
\special{ar 2720 2550 10 10 0  6.28318530717959E+0000}%
\special{sh 1}%
\special{ar 2720 2550 10 10 0  6.28318530717959E+0000}%
%
\special{pn 20}%
\special{sh 1}%
\special{ar 3620 2990 10 10 0  6.28318530717959E+0000}%
\special{sh 1}%
\special{ar 3620 3010 10 10 0  6.28318530717959E+0000}%
%
\special{pn 20}%
\special{sh 1}%
\special{ar 3640 850 10 10 0  6.28318530717959E+0000}%
\special{sh 1}%
\special{ar 3640 850 10 10 0  6.28318530717959E+0000}%
%
\special{pn 8}%
\special{pa 2620 2370}%
\special{pa 2588 2378}%
\special{pa 2556 2382}%
\special{pa 2526 2380}%
\special{pa 2496 2370}%
\special{pa 2468 2356}%
\special{pa 2442 2336}%
\special{pa 2416 2314}%
\special{pa 2400 2300}%
\special{sp}%
%
\special{pn 8}%
\special{pa 3780 2450}%
\special{pa 3814 2454}%
\special{pa 3846 2456}%
\special{pa 3878 2458}%
\special{pa 3910 2456}%
\special{pa 3940 2450}%
\special{pa 3970 2442}%
\special{pa 3998 2428}%
\special{pa 4026 2412}%
\special{pa 4052 2394}%
\special{pa 4078 2374}%
\special{pa 4104 2352}%
\special{pa 4130 2330}%
\special{pa 4130 2330}%
\special{sp}%
\end{picture}%
\hspace{0.5truecm}}} 

\vspace{1truecm}

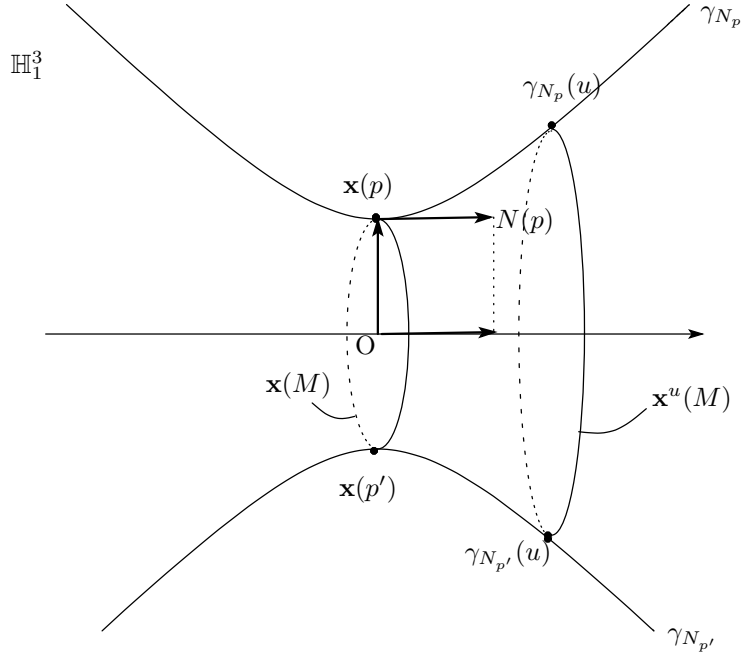
\centerline{Figure 3 : parallel surface}

\vspace{1truecm}


\vspace{0.5truecm}

\noindent
{\it Proof of Theorem \ref{thmD}.}\,
Let $M$ be a complex circle in $\mathbb{H}^3_1$, $\kappa$ be the radius of $M$. 
Let $\sqrt{\kappa}=c+\sqrt{-1}d$ $(c,d\in \mathbb{R})$. 
Since $M\subset \mathbb{H}^3_1$, we have $\mathrm{Re}(\kappa)=c^2-d^2=-1$.
We remember that a unit normal vector field $N$ of $M$ is given by $N(z)=(d+\sqrt{-1}c)(\cos z,\sin z)$.  
Hence the parallel surface $M^u$ of $M$ at distance $u$ is parameterized by 
\begin{eqnarray}
{\bf x}^u(z)&=& \cosh u\cdot (c+\sqrt{-1}d)(\cos z,\sin z)+\sinh u\cdot (d+\sqrt{-1}c)(\cos z, \sin z) \nonumber \\
&=& \kappa^u(\cos z, \sin z), \nonumber
\end{eqnarray}
where $\kappa^u$ is the complex number satisfying 
$\sqrt{\kappa^u}:=(c\cosh u+d\sinh u)+\sqrt{-1}(d\cosh u+c\sinh u)$.  
Thus $M^u$ is the complex circle of radius $\kappa^u$. 
It is easy to show that, when $u$ moves over $(-\infty,\infty)$, $\kappa^u$ moves over the whole of 
$\{ z\in \mathbb{C} \, |\, \mathrm{Re}(z)= -1 \}$. 
There the statement of Theorem \ref{thmD} follows from (\rnum{1}) and (\rnum{2}) of Theorem \ref{thmA}.
\qed

Next we prove Theorem C.

\vspace{0.5truecm}

\noindent
{\it Proof of Theorem C.}\, 
We consider a B-scroll $\displaystyle M(:\mathop{\Leftrightarrow_{\rm def}\,\, {\bf x}(s,t)=\gamma(s)+tB(s))}$ 
in $\mathbb{H}^3_1$. 
Since $N(s,t)=k_2tB(s)+C(s)$ where $k_2$ is the second curvature of $\gamma$, the parallel surface $M^u$ of $M$ is 
parameterized as 
$${\bf x}^u(s,t)=r(u)tB(s)+\sinh uC(s)+\cosh u\gamma(s).$$
We put $r(u):=\cosh u+k_2\sinh u$. 
By simple calculations, we have 
\begin{eqnarray}
\frac{\partial {\bf x}^u}{\partial s}&=&r(u)A(s)+k_1(s)\sinh uB(s)+r(u)t(-\gamma(s)+k_2C(s)), \nonumber \\
\frac{\partial {\bf x}^u}{\partial t}&=&r(u)B(s)\nonumber
\end{eqnarray}
and hence
\begin{eqnarray}
\begin{split}
\left \langle \frac{\partial {\bf x}^u}{\partial s},\frac{\partial {\bf x}^u}{\partial s}\right \rangle &=r(u)(-2k_2\sinh u+r(u)t^2(k_2^2-1)), \\ 
\left \langle \frac{\partial {\bf x}^u}{\partial s},\frac{\partial {\bf x}^u}{\partial t}\right \rangle &=-r(u)^2, \\
\left \langle \frac{\partial {\bf x}^u}{\partial t},\frac{\partial {\bf x}^u}{\partial t}\right \rangle &=0. 
\end{split}
\label{eq20}
\end{eqnarray}
If $r(u)= 0$, then ${\bf x}^u$ is not immersion. 
Hence, we need to assume that
\begin{eqnarray}
\begin{array}{ccc}
\mathrm{arctanh} \frac{-1}{k_2}<u<\infty & \mathrm{if}& k_2>1, \\
-\infty <u<\mathrm{arctanh} \frac{-1}{k_2} & \mathrm{if} & k_2<-1, \\
-\infty <u<\infty & \mathrm{if} & |k_2|\leq 1.
\end{array} \label{eq17}
\end{eqnarray}
The unit normal vector field $N^u$ of $M^u$ is given by 
$N^u(s,t)=\frac{k_2+r(u)\sinh u}{\cosh u}tB(s)+\cosh uC(s)+\sinh u\gamma(s)$.
Hence the shape operator $A_{N^u}$ in the direction $N^u$ is expressed with respect to 
the usual frame $(\frac{\partial {\bf x}^u}{\partial s},\frac{\partial {\bf x}^u}{\partial t})$ as
$$A_{N^u}=
\left(
\begin{array}{cc}
-\alpha & 0 \\
-\frac{k_1(s)}{r(u)^2} & -\alpha
\end{array}
\right)
,$$
where $\alpha=\frac{k_2+r(u)\sinh u}{r(u)\cosh u}$. 
When $\frac{\partial {\bf x}^u}{\partial s}$ is non-null, we put
$$e_1:=\frac{1}{\left|\frac{\partial {\bf x}^u}{\partial s}\right|}\frac{\partial {\bf x}^u}{\partial s}, 
\quad \tilde{e}_2:=\frac{1}{\left|\frac{\partial {\bf x}^u}{\partial s}\right|}\left(-r(u)^2 \frac{\partial {\bf x}^u}{\partial s}-
\left \langle \frac{\partial {\bf x}^u}{\partial s},\frac{\partial {\bf x}^u}{\partial s} \right \rangle \frac{\partial {\bf x}^u}{\partial t}\right).$$
We have $|\tilde{e}_2|=r(u)^2$. 
We put $e_2:=\frac{1}{r(u)^2}\tilde{e}_2$. 
Then, with respect to an orthonormal frame $(e_1,e_2)$, the shape operator $A_{N^u}$ is expressed as
$$A_{N^u}=
\left(
\begin{array}{cc}
-\alpha+\beta & -\beta \\
\beta & -\alpha-\beta
\end{array}
\right)
,$$
where $\beta=\frac{k_1(s)}{r(u)(-2k_2\sinh u+r(u)t^2(k_2^2-1))}$.
Thus, the mean curvature $\hat{\mathcal{H}}_u=-\alpha$ and $\Vert \hat{h_u}\Vert^2=2\alpha^2$.

When $\frac{\partial {\bf x}^u}{\partial s}$ is null, an orthonormal frame field $(e_1,e_2)$ on $M_u$ is given by 
$$e_1=\frac{1}{\sqrt{2}r(u)}\left(\frac{\partial {\bf x}^u}{\partial s}+\frac{\partial {\bf x}^u}{\partial t}\right), \quad
e_2=\frac{1}{\sqrt{2}r(u)}\left(\frac{\partial {\bf x}^u}{\partial s}-\frac{\partial {\bf x}^u}{\partial t}\right).$$
With respect to $(e_1, e_2)$, the shape operator $A_{N_u}$ is expressed as 
$$A_{N^u}=
\left(
\begin{array}{cc}
-\alpha-\beta^\prime & -\beta^\prime \\
\beta^\prime & -\alpha+\beta^\prime
\end{array}
\right)
,$$
where $\beta^\prime=\frac{k_1(s)}{2r(u)^2}$. 
Thus, the mean curvature $\hat{\mathcal{H}}_u=-\alpha$ and $\Vert \hat{h_u}\Vert^2=2\alpha^2$.
In both cases, it follows that $M_u$ has constant Gaussian curvature. 
Hence, by Fact \ref{KimKim}, $M_u$ is a B-scroll or a complex circle. 
By Theorem \ref{thmD}, $M_u$ is a B-scroll. 
By (\ref{eq5}), $M_u$ is flat if $\alpha^2=1$ and $M_u$ is non-flat if $\alpha^2\neq1$. \\
Assume that $M_u$ is flat for some $u\neq0$. 
Then we have $\alpha^2=1$ and hence $r(u)(-\sinh u\pm \cosh u)=k_2$. 
From this relation, we have $k_2=\pm 1$. 
Hence we obtain the second-half of statement of this theorem. 
\qed

\begin{remark}
We put 
\begin{eqnarray}
u_+:=
\left\{
\begin{array}{ll}
\infty & (k_2\geq -1)\\
\mathrm{arctanh} (\frac{-1}{k_2}) & (k_2<-1)
\end{array}
\right.
\quad \mathrm{and} \quad
u_-:=
\left\{
\begin{array}{ll}
-\infty & (k_2\leq 1)\\
\mathrm{arctanh} (\frac{-1}{k_2}) & (k_2>1).
\end{array}
\right.
\end{eqnarray}
If $u_+<\infty$ (resp. $u_->-\infty$), then $M_{u_+}$ (resp. $M_{u_-}$) is a focal submanifold of $M$ by (\ref{eq20}). \label{remark1}
\end{remark}

\begin{remark}
In $\mathbb{S}^3_1$, we can derive the following fact similar to Theorem \ref{thmB} and {\it Remark} \ref{remark1}. 
We put 
\begin{eqnarray}
u_+:=
\left\{
\begin{array}{ll}
\arctan \frac{-1}{k_2} & (k_2<0)\\
\frac{\pi}{2} & (k_2>0)\\
\end{array}
\right.
\quad
\mathrm{and} \quad 
u_-:=
\left\{
\begin{array}{ll}
-\frac{\pi}{2} & (k_2<0)\\
\arctan \frac{-1}{k_2} & (k_2>0),
\end{array}
\right.
\end{eqnarray}
$M_{u_+}$ and $M_{u_-}$ are a focal submanifold of $M$.
\end{remark}

\vspace{1truecm}

\centerline{
\small{
\unitlength 0.1in
%
%
\hspace{0.5truecm}}} 

\vspace{1truecm}

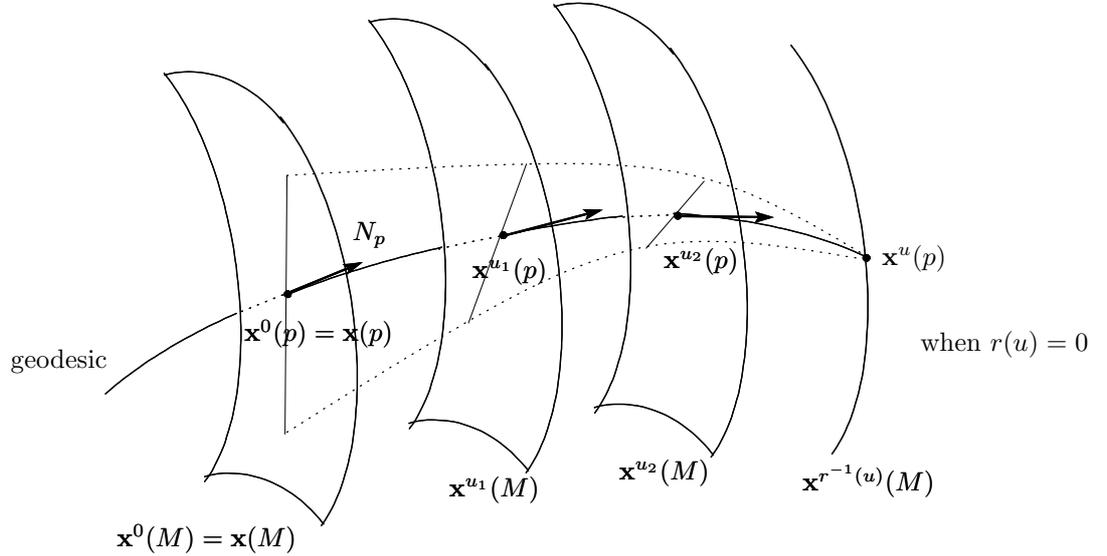
\centerline{Figure 4 : focal submanifold of B-scroll}

\vspace{1truecm}


\section{Cartan frame field of null curve of general order} 
E. Cartan showed that the existence of the Frenet type frame along a null curve of $\mathbb{R}^3_1$, called Cartan frame.                              
K. L. Duggal and A. Bejancu (\cite{DB}) constructed a Frenet type frame and equations along a null curve of a Lorentzian manifold. 
D. H. Jin (\cite{J}) constructed a new type Frenet frame and equations simpler than Duggal and Bejancu's Frenet frame and equations, by using it. 
Also, he proved that if we takes a special parameter (called distinguished parameter) then it becomes simpler moreover.
Daggal and Bejancu-type Frenet frame is called the general Frenet frame and Jin-type one is called the natural Frenet frame. 
Frenet equations also are named similary. 
In particular, Ferrandez-Gimenez-Lucas (\cite{FGL}) call Natural Frenet of null curves $\gamma$ parameterized by pseudo-arc parameter
(that is $\langle \nabla_{\dot{\gamma}}\dot{\gamma}, \nabla_{\dot{\gamma}}\dot{\gamma}\rangle =1$) the Cartan frame field. 
Note that the Cartan frame in this paper differ from the meaning of \cite{FGL}. 
\begin{lemm}
{\rm Let $\gamma(s)$ be a null curve of order $d (\geq 3)$ in $n$-dimensional Lorentzian manifold $(M,\langle , \rangle,\nabla)$, 
there is a frame field $(A,B,C,Z_1,Z_2,\cdots ,Z_{d-3})$ 
along $\gamma(s)$ satisfying 
\begin{eqnarray}
\begin{array}{l}
\langle A,A\rangle = \langle B,B\rangle=0,\quad \langle A,B\rangle=-1,\\ 
\langle A,C\rangle = \langle B,C\rangle=0,\quad \langle C,C\rangle=1,\\ 
\langle A,Z_i\rangle = \langle B,Z_i\rangle=\langle C,Z_i\rangle=0,\\ 
\langle Z_i,Z_j\rangle = \delta_{ij} 
\end{array}
\label{jouken2}
\end{eqnarray}
{\rm and}
\begin{eqnarray}
\left\{
\begin{array}{l}
\dot{\gamma}=A,\\ 
\nabla_{\dot{\gamma}} A=k_1C\\ 
\nabla_{\dot{\gamma}} C=k_2A+k_1B\\ 
\nabla_{\dot{\gamma}} B=k_2C+k_3Z_1\\ 
\nabla_{\dot{\gamma}} Z_1=k_3A+k_4Z_2\\ 
\nabla_{\dot{\gamma}} Z_2=-k_4Z_1+k_5Z_3\\ 
\hspace{50pt} \vdots \\ 
\nabla_{\dot{\gamma}} Z_{d-4}=-k_{d-2}Z_{d-5}+k_{d-1}Z_{d-3} \\ 
\nabla_{\dot{\gamma}} Z_{d-3}=-k_{d-1}Z_{d-4}. 
\end{array}
\right.
\label{jouken1}
\end{eqnarray}
This frame field $(A,B,C,Z_1,\cdots ,Z_{d-3})$ is called the {\it Cartan frame field along} $\gamma(s)$, 
where $k_i$ $(i=1,\cdots ,d-1)$ are positive-valued functions.
}
\label{defi2}
\end{lemm}
Natural Frenet equations of null curves parameterized by distinguished parameter and (\ref{jouken1}) equations are the same form by remarking $\langle A,B\rangle =-1$. 
The frame field in Lemma \ref{defi2} is different from the Natural Frenet frame field in constructive method and parameters condition.  
We construct in the order of $A,C,B,Z_1,\cdots, Z_{n-3}$, while they construct in the order of $A,B,C,Z_1,\cdots,Z_{n-3}$. 
Also, we constructed the frame field directly while they constructed the natural Frenet frame field by way of a general Frenet frame fields.  
In \cite{FGL}, the Cartan frame in the sense of their is constructed directly. 
Their constructive method and that of Lemma \ref{defi2} are the same except for 
a null curve parameterized by pseudo-arc or not. 
In this paper, a parameter of $\gamma$ is  arbitrary, that is, $\langle \nabla_{\dot{\gamma}}\dot{\gamma}, \nabla_{\dot{\gamma}}\dot{\gamma} \rangle$ is non-constant.
\begin{proof}
We constitue it like the Frenet formula.
First we show that
$\nabla_{\dot{\gamma}} A$ is non-null. Suppose that $\nabla_{\dot{\gamma}}A$ is null. 
Since $\gamma$ is of order $d (\geq 3)$, $A$ and $\nabla_{\dot{\gamma}}A$ are linearly independent. 
Denote by $W_1$ the $2$-dimensional subspace spanned by $A$ and $\nabla_{\dot{\gamma}}A$. Also we have
$$\langle A,\nabla_{\dot{\gamma}} A\rangle=\frac{1}{2}\dot{\gamma}\langle A,A\rangle=0.$$
Let $(e_1=A, e_2=\nabla_{\dot{\gamma}}A,e_3,\cdots ,e_{n})$ be a frame field along $\gamma(s)$. Then we have
$$(\langle e_i, e_j\rangle)=
\left(
\begin{array}{cc}
\begin{matrix}
0 & 0 \\
0 & 0
\end{matrix} & \text{\Large{\:*}} \\
 \text{\rule{0pt}{17pt}\Large{*}}  &  \text{\Large{*}}
\end{array}
\right).
$$
This contradicts the fact that $\langle , \rangle$ is Lorentzian. Therefore $\nabla_{\dot{\gamma}}A$ is non-null. \\
We put $k_1(s)=\sqrt{|\langle \nabla_{\dot{\gamma}}A,\nabla_{\dot{\gamma}}A\rangle|}$ and $C(s)=\frac{1}{k_1(s)}\nabla_{\dot{\gamma}}A$, 
and $\epsilon_C:=\langle C,C\rangle$. 
We have
\begin{eqnarray}
\nabla_{\dot{\gamma}}C=\left( \frac{1}{k_1}\right)^{{\huge \cdot}}
 \nabla_{\dot{\gamma}}A+ 
\frac{1}{k_1}\nabla_{\dot{\gamma}} (\nabla_{\dot{\gamma}}\dot{\gamma}) \quad 
\in \mathrm{Span}\{ \dot{\gamma}, \nabla_{\dot{\gamma}}\dot{\gamma}, \nabla_{\dot{\gamma}}(\nabla_{\dot{\gamma}}\dot{\gamma})\}
\label{nabraC}
\end{eqnarray}
Also we have $\langle \nabla_{\dot{\gamma}}C,C\rangle=0$ and $\langle A,C\rangle=0$, 
Hence
\begin{equation}
\langle \nabla_{\dot{\gamma}}C,A\rangle=-\langle C,\nabla_{\dot{\gamma}}A\rangle= -k_1\langle C,C\rangle \quad 
(\neq 0). \label{eq10}
\end{equation}  
Let $(\hat{e}_1=C, \hat{e}_2=A , \hat{e}_3, \cdots , \hat{e}_n)$ be a frame field along $\gamma$ 
satisfying $(\hat{e}_2, \cdots ,\hat{e}_n)$ is a frame field of $\mathrm{Span}\{ C\}^\bot$. Then we have \\
$$\left(\langle \hat{e}_i, \hat{e}_j \rangle \right)=
\left(
\begin{array}{ll}
\begin{matrix}
\epsilon_C & 0\\
0 & 0
\end{matrix} &
\begin{matrix}
0 & \cdots & 0\\
 & & \\
\end{matrix}\\
\hspace{3pt}
\begin{matrix}
0 & \\
\vdots & \\
0 & 
\end{matrix} 
& \hspace{15pt}\hsymb{*}
\end{array}
\right).
$$
Therefore we have $\epsilon_C=1$ and $\langle \nabla_{\dot{\gamma}}C,A\rangle = -k_1$ because $\langle , \rangle$ is Lorentzian. Hence we have 
Since $\gamma$ is of order $d (\geq 3)$, 
we have $\mathrm{dim}(\mathrm{Span}\{ A,\nabla_{\dot{\gamma}}C\})=2$. 
From these facts, it follows that $\mathrm{Span}\{ A, \nabla_{\dot{\gamma}}C\}$ is a $2$-dimensional Lorentzian space. 
There is a unique null vector $B \in \mathrm{Span}\{A,\nabla_{\dot{\gamma}}C\}$ such that $\langle A,B\rangle=-1$. 
It can be expressed as $\nabla_{\dot{\gamma}}C=aA+bB$ for some functions $a$ and $b$ because $\nabla_{\dot{\gamma}}C \in \mathrm{Span}\{A,B\}$. 
We have $b=k_1$ 
by (\ref{eq10}). We put $k_2:=a$. 
Put $W_2=\mathrm{Span}\{A,B,C\}^\bot$, 
which is $(n-3)$-dimensional Euclidean space. 
We put $$\nabla_{\dot{\gamma}}B=\hat{a}A+\hat{b}B+cC+Z\quad (Z\in W_2).$$ 
We have
\begin{eqnarray}
-\hat{a}&=& \langle \nabla_{\dot{\gamma}}B,B\rangle=0, \nonumber \\
-\hat{b}&=& \langle \nabla_{\dot{\gamma}}B,A\rangle=\langle B,\nabla_{\dot{\gamma}}A\rangle=-k_1\langle B,C\rangle=0,\nonumber \\
\hat{c}&=& \langle \nabla_{\dot{\gamma}}B,C\rangle=-\langle B,\nabla{\dot{\gamma}}C\rangle=k_2.\nonumber
\end{eqnarray}
If $d=3$, then we have $Z=0$. 
In the sequel, we consider the case of $d\geq 4$. \\
We put $k_3=|Z|$ and $Z_1=\frac{Z}{|Z|}$. 
Then $\nabla_{\dot{\gamma}}B$ can be expressed as $\nabla_{\dot{\gamma}}B=k_2C+k_3Z_1$. 
We put 
$$\nabla_{\dot{\gamma}}Z_1=\check{a}A+\check{b}B+\check{c}C+\check{z}_1Z_1+\hat{Z},$$
where $\hat{Z} \in \mathrm{Span}\{A, B, C, Z_1\}^\bot$. 
Then we have 
\begin{eqnarray}
-\check{b}&=& \langle \nabla_{\dot{\gamma}}Z_1,A\rangle=-\langle Z_1,k_1C\rangle=0, \nonumber \\
-\check{a}&=& \langle \nabla_{\dot{\gamma}}Z_1,B\rangle=-\langle Z_1, k_2C+k_3Z_1\rangle=-k_3, \nonumber \\
\check{c}&=& \langle \nabla_{\dot{\gamma}}Z_1, C\rangle=\langle Z_1,k_2A+k_1B\rangle=0, \nonumber \\
\check{z}_1&=& \langle \nabla_{\dot{\gamma}}Z_1,Z_1\rangle=0. \nonumber
\end{eqnarray} 
If $d=4$, then we have $\hat{Z}=0$. 
In the sequel, we consider the case of $d\geq 5$. Then we have $\hat{Z}\neq 0$.
We put $k_4=|\hat{Z}|$ and $Z_2=\frac{\hat{Z}}{|\hat{Z}|}$.
Then $\nabla_{\dot{\gamma}}Z_1$ can expressed as
$\nabla_{\dot{\gamma}}Z_1=k_3A+k_4Z_2$. 
We put 
$$\nabla_{\dot{\gamma}}Z_2=\tilde{a}A+\tilde{b}B+\tilde{c}C+\tilde{z}_1Z_1+\tilde{z}_2Z_2+\tilde{Z},$$
where $\tilde{Z} \in \mathrm{Span}\{ A, B, C, Z_1,Z_2\}^\bot$. 
Then we have 
$\tilde{a}=\tilde{b}=\tilde{c}=\tilde{z}_2=0$ and $\tilde{z}_1=-k_4$. 
If $d=5$, then we have $\tilde{Z}=0$. 
In the sequel, we consider the case of $d\geq 6$. 
Then we have $\tilde{Z}\neq0$. 
We put $k_5=|\tilde{Z}|$ and $Z_3=\frac{\tilde{Z}}{|\tilde{Z}|}$. 
Then $\nabla_{\dot{\gamma}}Z_2$ is expressed as $\nabla_{\dot{\gamma}}Z_2=-k_4Z_1+k_5Z_3.$
By repeating the same discussion, we can derive the relations in Lemma \ref{defi2}. 
\end{proof}

It has already been proven their uniquely exist of a null curve of $\mathbb{R}^n_1$ 
equipped with a general Frenet frame field for differentiable for any given functions $k_i :  [-\epsilon ,\epsilon ] \rightarrow \mathbb{R}$. 
For a null curve of $\mathbb{S}^n_1$ (or $\mathbb{H}^n_1$) and a Cartan frame field in the sense of \cite{FGL}, respectively. 
See \cite{DB}, \cite{DJ}, \cite{FGL} etc about these facts. 
We prove the above about the Cartan frame field in Lemma \ref{defi2} in another method.
\begin{prop}
{\rm
Let $k_1, \cdots , k_n : [-\epsilon,\epsilon]\rightarrow \mathbb{R}$ be differentiable non-zero functions. 
There is a unique null curve $\gamma$ of order $n$ of $\mathbb{R}^n_1$ equipped with a frame field $(A, B, C, Z_1, \cdots, Z_{n-3})$ along $\gamma$ 
satisfying (\ref{jouken2}) and (\ref{jouken1}).
}\label{lemmCFNC}
\end{prop}
\begin{proof}
We put $V:=\frac{1}{\sqrt{2}}(A+B)$ and $W:=\frac{1}{\sqrt{2}}(A-B)$. 
Let $\tilde{F}$ be a matrix consisting column vector fields $V, W, C, Z_1, \cdots, Z_{n-3}$, that is 
$\tilde{F}=(V \quad W \quad C \quad Z_1 \quad \cdots \quad Z_{n-3})$. 
We put 
\begin{eqnarray}
\tilde{K}=
\left(
\begin{array}{lll}
\begin{matrix}
0 &0 & \frac{1}{\sqrt{2}}(k_1+k_2) & \frac{1}{\sqrt{2}}k_3 &0 \\ 
0 &0 & -\frac{1}{\sqrt{2}}(k_1-k_2) & \frac{1}{\sqrt{2}}k_3 &0 \\
\frac{1}{\sqrt{2}}(k_1+k_2) & \frac{1}{\sqrt{2}}(k_1-k_2) &0 &0 &0 \\
\frac{1}{\sqrt{2}}k_3 & -\frac{1}{\sqrt{2}}k_3 &0 &0 &-k_4 \\
0 &0 &0 &k_4 &0 & 
\end{matrix}
&
&
\\
&
\ddots
&
\\
& &
\begin{matrix}
0 & -k_{n-1} &0 \\
k_{n-1} & 0 & -k_n \\
0& k_n & 0
\end{matrix}
\end{array}
\right) , \nonumber
\end{eqnarray}
then (\ref{jouken1}) holds
\begin{eqnarray}
\dot{\tilde{F}}=\tilde{F}\tilde{K}. \label{ODE2}
\end{eqnarray} 
If you give the initial value, the solution is unique. We prove that the solution satisfying (\ref{jouken2}). 
We put $E:=\mathrm{diag}(-1,1,\cdots,1)$.
Let $\Phi$ be (\ref{ODE2}) solution with $E$ as the initial value. 
It is easy to check that for any initial value $\tilde{F}_0$ the solution expressed as $\tilde{F}=\tilde{F}_0  E\Phi$
and $E\tilde{K}$ is skew-symmetric.
Hence
\begin{eqnarray}
\frac{d}{ds}(E\tilde{F}^{-1}E)(s)&=&-E\tilde{K}(s)\tilde{F}^{-1}(s)E \nonumber \\
                              &=& ^t(E\tilde{K}(s))\tilde{F}^{-1}(s)E \nonumber \\
                              &=& ^t\tilde{K}(s)(E\tilde{F}^{-1}E)(s). \nonumber
\end{eqnarray} 
By
\begin{eqnarray}
^t\Phi(s_0) &=& ^tE =E \nonumber \\
(E\Phi^{-1} E)(s_0) &=& E, \nonumber
\end{eqnarray}
we obtain $^t\tilde{F}=E\tilde{F}^{-1}E$. 
Hence, the columns of $\tilde{F}$ form an orthonormal basis for $\mathbb{R}^n_1$ and $V$ is timelike.
Thus $F$ satisfies (\ref{jouken2}). 
When we put $\gamma(s)=\int^s_{s_0} A(t)dt$, $\gamma$ is null curve satisfying (\ref{jouken2}) and (\ref{jouken1}).\\
Lastly we prove the uniqueness of $\gamma$.
Let $\gamma_1$ and $\gamma_2$ be null curves that have same Cartan curvatures $k_1,\cdots,k_n$.
Let $\tilde{F}_i$ be an pseudo-orthogonal matrix defined $\tilde{F}_i=(V^i \quad W^i \quad C^i \quad Z^i_1 \quad \cdots \quad Z^i_{n-3})$ for $\gamma_i(s)$ ($i=1,2$).
We put $\tilde{F}_1^0:=\tilde{F}_1(s_0)$ and $\tilde{F}_2^0=\tilde{F}_2(s_0)$.
When we put $L:=\tilde{F}_2^0(\tilde{F}_1^0)^{-1}$, relationship between $\tilde{F}_1$ and $\tilde{F}_2$ is $\tilde{F}_1(s)=L\tilde{F}_2(s)$. 
In particular, $A^2(s)=LA^2(s)$ because $V^2(s)=LV^1(s)$ and $W^2(s)=LW^1(s)$.
We put $b:=\gamma_2(s_0)-L\gamma_1(s_0)$, we have
\begin{eqnarray}
\gamma_2(s)&=& \gamma_2(s_0)+\int_{s_0}^s A^2(t)dt \nonumber \\
           &=& L\gamma_1(s_0)+b+\int_{s_0}^s LA^1(t)dt \nonumber \\
           &=& L\gamma_1(s)+b. \nonumber 
\end{eqnarray}           
Since $L$ is a pseudo-orthogonal matrix, it is linear isometry.                         
\end{proof}
\begin{remark}
{\rm In the case of $d=3$, we have 
\begin{eqnarray}
\left\{
\begin{array}{l}
\dot{\gamma}=A \nonumber \\
\nabla_{\dot{\gamma}}A=k_1C \nonumber \\
\nabla_{\dot{\gamma}}C=k_2A+k_1B \nonumber \\
\nabla_{\dot{\gamma}}B=k_2C. \nonumber 
\end{array}
\right.
\end{eqnarray}
In the case of $d=4$, we have
\begin{eqnarray}
\left\{
\begin{array}{l}
\dot{\gamma}=A \nonumber \\
\nabla_{\dot{\gamma}}A=k_1C \nonumber \\
\nabla_{\dot{\gamma}}C=k_2A+k_1B \nonumber \\
\nabla_{\dot{\gamma}}B=k_2C+k_3Z_1 \nonumber \\
\nabla_{\dot{\gamma}}Z_1=k_3A. \nonumber 
\end{array}
\right.
\end{eqnarray}
}
\end{remark}

\begin{lemm}
{\rm 
Let $\gamma$ be a null curve $\gamma(s)$ of order $d  (\geq3)$ in 
$\mathbb{H}^{n}_{1}$ (resp. $\mathbb{S}^n_1$) and  
 $(A,B,C,Z_1,\cdots ,Z_{n-3})$ the Cartan frame field along $\gamma(s)$. 
Then, with respect to the connection $\tilde{\nabla}$ of $\mathbb{E}^{n+1}_2$ (resp. $\mathbb{E}^{n+1}_1$), 
 the following relations hold: }
\begin{eqnarray}
\left\{
\begin{array}{l}
\tilde{\nabla}_{\dot{\gamma}} A=k_1C \\
\tilde{\nabla}_{\dot{\gamma}} C=k_2A+k_1B \\
\tilde{\nabla}_{\dot{\gamma}} B=k_2C+k_3Z_1+\epsilon \gamma  \\
\tilde{\nabla}_{\dot{\gamma}} Z_1=k_3A+k_4Z_2 \\
\tilde{\nabla}_{\dot{\gamma}} Z_i=-k_{i+2}Z_{i-1}+k_{i+3}Z_{i+1} \quad (2\leq i\leq n-4)\\
\tilde{\nabla}_{\dot{\gamma}} Z_{n-3}=-k_{n-1}Z_{n-4}
\end{array}
\right.
\label{jouken3}
\end{eqnarray}
{\rm where $\epsilon=-1$ (resp. $\epsilon=+1$)}.
\label{lemm7.2}
\end{lemm}
\begin{proof}
Let $\hat{H}_{\bf x}$ be the mean curvature vector of $\mathbb{H}^n_1$ in $\mathbb{E}^{n+1}_{2}$ 
(resp. $\mathbb{S}^n_1$ in $\mathbb{E}^{n+1}_1$).  
Then we have $\hat{H}_\gamma=-\epsilon \gamma$.  
Since $\mathbb{H}^n_1$ (resp. $\mathbb{S}^n_1$) is totally umbilic in $\mathbb{E}^{n+1}_2$ 
(resp. $\mathbb{E}^{n+1}_1$), we have 
\begin{eqnarray}
\left\{
\begin{array}{l}
\tilde{\nabla}_{\dot{\gamma}} A=k_1C+\langle A,A\rangle \hat{H}_\gamma=k_1C\nonumber \\
\tilde{\nabla}_{\dot{\gamma}} C=k_2A+k_1B+\langle A,C\rangle \hat{H}_\gamma=k_2A+k_1B \nonumber \\
\tilde{\nabla}_{\dot{\gamma}} B=k_2C+k_3Z_1+\langle A,B\rangle \hat{H}_\gamma=k_2C+k_3Z_1+\epsilon \gamma \nonumber\\
\tilde{\nabla}_{\dot{\gamma}} Z_1=k_3A+k_4Z_2+\langle A,Z_1\rangle \hat{H}_\gamma=k_3A+k_4Z_2\nonumber \\
\tilde{\nabla}_{\dot{\gamma}} Z_i=-k_{i+2}Z_{i-1}+k_{i+3}Z_{i+1}+\langle A,Z_i\rangle \hat{H}_\gamma=-k_{i+2}Z_{i-1}+k_{i+3}Z_{i+1}\nonumber \\
\tilde{\nabla}_{\dot{\gamma}} Z_{n-3}=-k_{n-1}Z_{n-4}+\langle A,Z_{n-3}\rangle \hat{H}\gamma=-k_{n-1}Z_{n-4}. \nonumber
\end{array}
\right.
\end{eqnarray}
\end{proof}
\begin{prop}
{\rm
Let $k_1, \cdots , k_n : [-\epsilon,\epsilon] \rightarrow \mathbb{R}$ be differentiable non-zero functions. 
There is a unique null curve $\gamma$ of order $n$ of $\mathbb{H}^n_1$ (or $\mathbb{S}^n_1$) 
equipped with a frame field $(A, B, C, Z_1, \cdots, Z_{n-3})$ along $\gamma$ 
satisfying $\dot{\gamma}=A$, (\ref{jouken2}), (\ref{jouken3}) and
\begin{eqnarray}
\begin{array}{l}
\langle \gamma, \gamma \rangle =\epsilon \\
\langle A, \gamma \rangle =\langle B,\gamma \rangle =\langle C, \gamma \rangle =\langle Z_i,\gamma \rangle =0
\end{array}
\label{jouken4}
\end{eqnarray}
where $\epsilon=-1$ (resp. $\epsilon=+1$).
}
\end{prop}
\begin{proof}
We put $V=\frac{1}{\sqrt{2}}(A+B)$, $W=\frac{1}{\sqrt{2}}(A-B)$ and $\tilde{F}=(\gamma \quad V\quad W\quad Z_1\quad \cdots \quad Z_{n-3})$. 
Then (\ref{jouken3}) holds
\begin{eqnarray}
\dot{\tilde{F}}(s)=\tilde{F}(s)\tilde{K}(s) \label{ODE3}
\end{eqnarray}
where
\begin{eqnarray}
\tilde{K}=
\left(
\begin{array}{lll}
\begin{matrix}
0 & \frac{1}{\sqrt{2}}\epsilon & -\frac{1}{\sqrt{2}}\epsilon & 0& 0 \\
\frac{1}{\sqrt{2}}& 0 &0 & \frac{1}{\sqrt{2}}(k_1+k_2) & \frac{1}{\sqrt{2}}k_3 & \\ 
\frac{1}{\sqrt{2}} &0 &0 & -\frac{1}{\sqrt{2}}(k_1-k_2) & \frac{1}{\sqrt{2}}k_3 & \\
0& \frac{1}{\sqrt{2}}(k_1+k_2) & \frac{1}{\sqrt{2}}(k_1-k_2) &0 &0 & \\
0& \frac{1}{\sqrt{2}}k_3 & -\frac{1}{\sqrt{2}}k_3 &0 &0 &-k_4 \\
0& 0 &0 &0 &k_4 &0 & 
\end{matrix}
&
&
\\
&
\ddots
&
\\
& &
\begin{matrix}
0 & -k_{n-1} &0 \\
k_{n-1} & 0 & -k_n \\
0& k_n & 0
\end{matrix}
\end{array}
\right) . \nonumber
\end{eqnarray}
If you give the initial value, (\ref{ODE3}) solution is unique. 
Let $E$ be $\mathrm{diag}(\epsilon, -1,1,\cdots,1)$.
By simple calculations, we have $^t\tilde{K}=-E\tilde{K}E$. Hence
\begin{eqnarray}
\frac{d}{ds}(E\tilde{F}^{-1}E)(s)&=& -E\tilde{K}(s)\tilde{F}(s)^{-1}E \nonumber \\
                                 &=& -E\tilde{K}(s)E(E\tilde{F}^{-1}E)(s)\nonumber \\
                                 &=& ^t\tilde{K}(s)(E\tilde{F}^{-1}E)(s).\nonumber
\end{eqnarray}
Let $\Phi$ be (\ref{ODE3}) solution with $E$ as the initial value, then $^t\Phi(s_0)=(E\Phi^{-1}E)(s_0)=E$.
Thus $^t\tilde{F}=E\tilde{F}^{-1}E$ and $(\gamma, A, B, C, Z_1, \cdots, Z_{n-3})$ satisfies (\ref{jouken2}) and (\ref{jouken4}). 
\end{proof}
\section{Proof of Theorem \ref{thmC}.}

In this section, we prove Theorem \ref{thmC}.  

\noindent
{\it Proof of Theorem \ref{thmC}.} 
By Lemma \ref{lemm7.2}, we have
\begin{eqnarray}
\begin{array}{l}
{\bf x}_s = A(s)-{\displaystyle \frac{k_1(s) k_2 |z|^2}{2}}B(s) +k_2tC(s) +k_3tZ_1(s) +{\displaystyle \sum_{j=1}^{n-2}z_j\dot{Z}_j(s)} -t\gamma(s), \\
{\bf x}_t = B(s), \vspace{5pt}\\
{\bf x}_{z_i} = z_i\gamma(s) + Z_i(s) -k_2z_iC(s).
\end{array}
\end{eqnarray}
The unit normal vector field $N$ of the Lorentzian hypersurface ${\bf x}:M^n_1\hookrightarrow \mathbb{H}^{n+1}_1$ as in Theorem \ref{thmC} is given by
$$N(s,t,z)=k_2tB(s)+k_2\sum_{j=1}^{n-2}z_jZ_j(s)+\left(1-\frac{|z|^2}{2}\right)C(s)+\frac{k_2|z|^2}{2}\gamma(s),$$
where $z=(z_1, \cdots ,z_{n-2})$.  
With respect to $(\frac{\partial{\bf x}}{\partial s},\frac{\partial{\bf x}}{\partial t},
\frac{\partial {\bf x}}{\partial z_1},\cdots, \frac{\partial {\bf x}}{\partial z_{n-2}})$, 
the shape operator $A_N$ is expressed as 
$$A_N=\left(
\begin{array}{ll}
\begin{matrix}
-k_2 & 0 \\
-k_1(s) & -k_2
\end{matrix}
& 
\hspace{30pt}\BigZero
\\
\hspace{30pt}\BigZero
&
\begin{matrix} 
-k_2 & & \\
     & \ddots & \\
 & & -k_2
\end{matrix}
\end{array}
\right)
.$$
By Lemma \ref{lemm7.2}, we obtain
\begin{eqnarray}
\begin{array}{ll}
\langle \dot{Z}_1,\dot{Z}_1\rangle =k_4^2 & \\
\langle \dot{Z}_i,\dot{Z}_i\rangle =k_{i+2}^2+k_{i+3}^2 & (2\leq i\leq n-3)\\
\langle \dot{Z}_{n-3},\dot{Z}_{n-3}\rangle =k_n^2 & \\
\langle \dot{Z}_i,\dot{Z}_j\rangle =0 & (i\neq j \quad \mathrm{or} \quad i\neq j\pm2)\\
\langle \dot{Z}_i,\dot{Z}_{i+2}\rangle =-k_{i+3}k_{i+4} & (1\leq i\leq n-4)
\end{array}
\end{eqnarray}
and
\begin{eqnarray}
\sum_{i,j=1}^{n-2} z_iz_j \langle \dot{Z_i},\dot{Z_j} \rangle &=&k_4^2z_1^2+\sum_{i=2}^{n-3}(k_{i+2}^2+k_{i+3}^2)z_i^2+k_n^2z_{n-3}^2 +2(-k_4k_5z_1z_3-\cdots -k_{n-1}k_nz_{n-4}z_{n-2}) \nonumber \\
                                                              &=&k_4^2z_1^2+k_n^2z_{n-3}^2 +\sum_{i=1}^{n-3}(k_{i+2}z_i-k_{i+3}z_{i+1})^2. \nonumber
\end{eqnarray}
Hence
\begin{flalign}
\langle {\bf x}_s,{\bf x}_s\rangle &= -k_1k_2|z|^2\langle A,B\rangle +k_2^2t^2\langle C,C\rangle + k_3^2t^2\langle Z_1,Z_1\rangle +\sum_{i,j=1}^{n-2}z_iz_j\langle \dot{Z}_i,\dot{Z}_j\rangle \nonumber \\
                                   & \hspace{20pt} +t^2\langle \gamma, \gamma \rangle  +2k_3t\langle Z_1, \sum_{j=1}^{n-2}z_j\dot{Z_j}\rangle -k_1k_2|z|^2\langle B, z_1\dot{Z}_1\rangle \nonumber \\
                                   &= k_1k_2|z|^2(1+k_3z_1)+k_3^2t^2+k_4^2z_1^2 +\sum_{i=1}^{n-3}(k_{i+2}z_i-k_{i+3}z_{i+1})^2 +k_n^2z_{n-3}^2 -2k_3k_4z_2t,\nonumber \\
\langle {\bf x}_s,{\bf x}_{z_i}\rangle &= k_3t\delta_{1i} +\sum_{j=1}^{n-2}z_j\langle Z_i,\dot{Z}_j\rangle \nonumber \\
&=
\left\{
\begin{array}{l}
k_3t-k_4z_2 \quad (i=1)\\
k_{i+2}z_{i-1}-k_{i+3}z_{i+1} \quad (2\leq i\leq n-3) \\
k_nz_{n-3} \quad (i=n-2)
\end{array}
\right. ,\nonumber \\
\langle {\bf x}_s,{\bf x}_t\rangle &=-(1+|z|^2+k_3z_1),\quad\,\,
\langle {\bf x}_t,{\bf x}_t\rangle = \langle {\bf x}_t,{\bf x}_{z_i}\rangle =0,\quad\,\,
\langle {\bf x}_{z_i},{\bf x}_{z_j}\rangle = \delta_{ij}. \nonumber
\end{flalign}
We put $e_j={\bf x}_{z_j}$ $(1\leq j\leq n-2)$ 
and  
$$\tilde{e}_{n-1}:={\bf x}_s-\sum_{j=1}^{n-2}\langle {\bf x}_s,e_j\rangle e_j.$$
If $\tilde{e}_{n-1}$ is non-null, we put 
$$e_{n-1}:=\frac{1}{|\tilde{e}_{n-1}|}\tilde{e}_{n-1},\quad 
\tilde{e}_n:=\frac{1}{|\tilde{e}_{n-1}|}(\langle {\bf x}_s,{\bf x}_t\rangle e_{n-1} -\langle \tilde{e}_{n-1},\tilde{e}_{n-1}\rangle {\bf x}_t)$$
and $e_n:=\frac{1}{\langle {\bf x}_s,{\bf x}_t\rangle}\tilde{e}_n$. 
With respect to the orthonormal frame field $(e_1,\cdots ,e_n)$, the shape operator $A_N$ is expressed as
$$A_N=
\left(
\begin{array}{ll}
\begin{matrix}
-k_2 & & \\
 & \ddots & \\
 & & -k_2 
\end{matrix}
&
\hspace{40pt}\BigZero
\\
\hspace{30pt}\BigZero
&
\begin{matrix} 
-k_2-\alpha & -\alpha \\
 \alpha & -k_2+\alpha
\end{matrix}
\end{array}
\right)
,$$
where $\alpha=\epsilon_{n-1}\frac{k_1|\tilde{e}_n|}{|\tilde{e}_{n-1}|^2}$.
If $\tilde{e}_{n-1}$ is null, we put 
$$e_{n-1}:=\frac{1}{\sqrt{2}\langle \tilde{e}_{n-1},{\bf x}_t\rangle ^{\frac{1}{2}}}(\tilde{e}_{n-1}+{\bf x}_t),$$
$$e_{n}:=\frac{1}{\sqrt{2}\langle \tilde{e}_{n-1},{\bf x}_t\rangle ^{\frac{1}{2}}}(\tilde{e}_{n-1}-{\bf x}_t).$$
With the orthonormal frame field $(e_1,\cdots ,e_n)$, the shape operator $A_N$ is expressed as 
$$A_N=
\left(
\begin{array}{cc}
\begin{matrix}
-k_2 & & \BigZero \\
 & \ddots & \\
 & & -k_2 
\end{matrix}
&
\begin{matrix}
 \beta_1 & \beta_1\\
 \vdots & \vdots \\
 \beta_{n-2} & \beta_{n-2}
\end{matrix} 
\\
\hspace{0pt}\BigZero
&
\begin{matrix}
-k_2-\alpha & -\alpha \\
\alpha & -k_2+\alpha
\end{matrix}
\end{array}
\right)
,$$
where $\alpha=\frac{k_1(s)}{2}$ and $\beta_j=\frac{k_2}{\sqrt{2}\langle \tilde{e}_{n-1},{\bf x}_t\rangle^\frac{1}{2}}\sum_{j=1}^{n-2}\langle {\bf x}_s,e_j\rangle$. 
In both cases, $\Vert \hat{h}\Vert^2=n$ and the mean curvature $\hat{\mathcal{H}}=-k_2$. 
By (\ref{eq5}), $M^n_1$ is scalar flat. 
Hence, $\Delta \tilde{\nu}=n\tilde{\nu}-2k_2N\wedge e_1\wedge e_2 \neq0$ and $\Delta^2\tilde{\nu}=0$ by (\ref{eq1}) and (\ref{eq3}). 
Therefore, $\tilde{\nu}$ is of infinite type by Lemma {\ref{lemm1}}.\qed

\vspace{0.5truecm}

{\small 
\leftline{Honoka Kobayashi}
\leftline{Department of Mathematics, Faculty of Science}
\leftline{Tokyo University of Science, 1-3 Kagurazaka}
\leftline{Shinjuku-ku, Tokyo 162-8601 Japan}
\leftline{{\it e-mail}: 1116607@ed.tus.ac.jp} 
}

\vspace{5pt}

{\small 
\leftline{Naoyuki Koike}
\leftline{Department of Mathematics, Faculty of Science}
\leftline{Tokyo University of Science, 1-3 Kagurazaka}
\leftline{Shinjuku-ku, Tokyo 162-8601 Japan}
\leftline{{\it e-mail}: koike@rs.kagu.tus.ac.jp}
}

\end{document}